\documentclass[a4paper,1p,number,sort&compress,times]{elsarticle}
\usepackage{amsthm,amssymb}
\usepackage{graphicx,psfrag,subfigure}

\newcommand{\Cset}{\mathbb{C}}

\newcommand{\Qset}{\mathbb{Q}}
\newcommand{\Rset}{\mathbb{R}}
\newcommand{\Zset}{\mathbb{Z}}

\theoremstyle{plain}
\newtheorem{thm}{Theorem}
\newtheorem{pro}[thm]{Proposition}
\newtheorem{lem}[thm]{Lemma}
\newtheorem{cor}[thm]{Corollary}
\newtheorem{con}{Conjecture}
\newdefinition{defi}{Definition}
\newdefinition{remark}{Remark}
\newdefinition{example}{Example}

\newcommand{\rmE}{{\rm E}}
\newcommand{\rmH}{{\rm H}}
\newcommand{\rme}{{\rm e}}
\newcommand{\rmd}{\mskip2mu{\rm d}\mskip1mu}
\newcommand{\rmm}{{\rm m}}
\newcommand{\rmp}{{\rm p}}
\newcommand{\HS}{{{\rm hom},{\rm sym}}}
\newcommand{\Symmetric}{{\rm sym}}
\newcommand{\Rank}{\mathop{\rm rank}\nolimits}

\newcommand{\Divide}[2]{{#1}\,\big| \, {#2}}
\newcommand{\Order}{\mathop{\rm O}\nolimits}

\begin{document}

\begin{frontmatter}

\title{On Cayley conditions for billiards inside ellipsoids\tnoteref{t1}}

\tnotetext[t1]{Research supported in part by
MICINN-FEDER grants MTM2009-06973 and MTM2012-31714 (Spain)
and CUR-DIUE grant 2009SGR859 (Catalonia).
Useful conversations with Pablo S. Casas, Yuri Fedorov, and Bernat Plans
are gratefully acknowledged.}

\author{Rafael Ram\'{\i}rez-Ros}

\ead{Rafael.Ramirez@upc.edu}

\address{Departament de Matem\`{a}tica Aplicada I,
         Universitat Polit\`{e}cnica de Catalunya,
         Diagonal 647, 08028 Barcelona, Spain}

\begin{abstract}
All the segments (or their continuations) of a billiard trajectory
inside an ellipsoid of $\Rset^n$ are tangent to $n-1$
quadrics of the pencil of confocal quadrics determined by the ellipsoid.
The quadrics associated to periodic billiard trajectories verify
certain algebraic conditions.
Cayley found them in the planar case.
Dragovi\'{c} and Radnovi\'{c} generalized them to any dimension.
We rewrite the original matrix formulation of these
\emph{generalized Cayley conditions} as a simpler polynomial one.
We find several remarkable algebraic relations between caustic parameters
and ellipsoidal parameters that give rise to nonsingular periodic trajectories.
\end{abstract}

\begin{keyword}
billiard \sep integrable system \sep periodic trajectory \sep Cayley condition
\MSC 37J20 \sep 37J35 \sep 37J45 \sep 70H06 \sep 14H99
\end{keyword}

\end{frontmatter}

\section{Introduction}

One of the best known discrete integrable system is the billiard inside ellipsoids.
All the segments (or their continuations) of a billiard trajectory
inside an ellipsoid of $\Rset^n$ are tangent to
$n-1$ quadrics of the pencil of confocal quadrics determined
by the ellipsoid~\cite{KozlovTreschev1991,Tabachnikov1995,Tabachnikov2005}.
This situation is fairly exceptional.
Quadrics are the only smooth hypersurfaces of $\Rset^n$, $n\ge 3$,
that have caustics~\cite{Berger1995,Gruber1995}.
A \textit{caustic} is a smooth hypersurface with the property
that a billiard trajectory, once tangent to it,
stays tangent after every reflection.
Caustics are a geometric manifestation of the integrability of
billiards inside ellipsoids.

Periodic trajectories are the most distinctive trajectories,
so their study is the first task.
There exist two remarkable results about periodic billiard trajectories
inside ellipsoids:
the \textit{generalized Poncelet theorem} and
the \textit{generalized Cayley conditions}.

A classical geometric theorem of Poncelet~\cite{Poncelet1822,GriffichsHarris1977}
implies that if a billiard trajectory inside an ellipse is periodic,
then all the trajectories sharing its caustic are also periodic.
Its generalization to the spatial case was proved by Darboux~\cite{Darboux1887}.
The extension of this result to arbitrary dimensions can be found
in~\cite{ChangFriedberg1988,Chang_etal1993a,Chang_etal1993b,Previato1999}.
The generalized Poncelet theorem can be stated as follows.
If a billiard trajectory inside an ellipsoid is closed after $m_0$ bounces
and has length $L_0$,
then all trajectories sharing the same caustics are also closed
after $m_0$ bounces and have length $L_0$.
Thus, a natural question arises.
What caustics do give rise to periodic trajectories?
The planar case was solved by Cayley~\cite{Cayley1854,GriffichsHarris1978}
in the XIX century.
Dragovi\'{c} and Radnovi\'{c}~\cite{DragovicRadnovic1998a,DragovicRadnovic1998b}
found some \emph{generalized Cayley conditions} for billiards inside ellipsoids
fifteen years ago.
They have also stated similar conditions in other billiard frameworks;
see~\cite{DragovicRadnovic2004,DragovicRadnovic2006,DragovicRadnovic2006b,
DragovicRadnovic2008,DragovicRadnovic2012}.

For simplicity, let us focus on the spatial case.
Let $Q : x^2/a + y^2/b + z^2/c = 1$ be the triaxial ellipsoid
with ellipsoidal parameters $0 < c < b < a$.
Any billiard trajectory inside $Q$ has as caustics two elements
of the family  of confocal quadrics
\[
Q_\lambda =
\left\{
(x,y,z) \in \Rset^3 :
\frac{x^2}{a-\lambda} + \frac{y^2}{b-\lambda} + \frac{z^2}{c-\lambda} = 1
\right\}.
\]
We restrict our attention to nonsingular trajectories.
That is, trajectories with two different caustics
which are ellipsoids: $0 < \lambda < c$,
hyperboloids of one sheet: $c < \lambda < b$,
or hyperboloids of two sheets: $b < \lambda < a$.
The singular values $\lambda \in \{a,b,c\}$ are discarded.
There exist some restrictions on the caustics $Q_{\lambda_1}$ and $Q_{\lambda_2}$.
It is known that EH1, H1H1, EH2, and H1H2 are the only
feasible caustic types; see~\cite{Knorrer1980,Audin1994}.
The meaning of these notations is evident.

The generalized Cayley condition in this context
can be expressed as follows.
The billiard trajectories inside the triaxial ellipsoid $Q$
sharing the caustics $Q_{\lambda_1}$ and $Q_{\lambda_2}$
are periodic with \emph{elliptic period} $m \ge 3$ if and only if
\[
\Rank
\left(
\begin{array}{ccc}
f_{m + 1}  & \cdots & f_4 \\
\vdots   &        & \vdots \\
f_{2m - 1} & \cdots & f_{m + 2}
\end{array}
\right) < m - 2,
\]
where $f(t) = \sum_{l\ge 0} f_l t^l :=
\sqrt{(1-t/a)(1-t/b)(1-t/c)(1-t/\lambda_1)(1-t/\lambda_2)}$.
We claim that most results related to generalized Cayley conditions
become simpler when expressed in terms of the inverse quantities
$1/a$, $1/b$, $1/c$, $1/\lambda_1$, and $1/\lambda_2$.
Proposition~\ref{pro:C_n_n} is a paradigmatic example.

The elliptic period is defined in Section~\ref{sec:Preliminaries}.
Roughly speaking,
the difference between the period $m_0$ and the elliptic period $m$
of the periodic billiard trajectories sharing two given caustics
is that \emph{all} those trajectories close
in Cartesian (respectively, elliptic) coordinates after exactly
$m_0$ (respectively, $m$) bounces.
We will see that either $m = m_0/2$ or $m = m_0$.

The previous matrix formulation is very nice from a theoretical point of view,
but it has strong limitations from a computational point of view.
We will see in Section~\ref{sec:Practical} that
it can be written as a system of two homogeneous symmetric polynomial
equations with rational coefficients of degrees
$m^2 -2$ and $m^2 - 1$ in the variables
$1/a$, $1/b$, $1/c$, $1/\lambda_1$, and $1/\lambda_2$.
Thus, both degrees grow \emph{quadratically} with the elliptic period $m$,
which turns this approach into a tough challenge.
In particular, to our knowledge,
the caustic parameters $\lambda_1$ and $\lambda_2$
have never been explicitly expressed in terms of the ellipsoidal parameters
$a$, $b$, and $c$ for any $m \ge 3$.

We will rewrite this matrix formulation as a computationally more appealing one
which gives rise to (non-symmetric) homogeneous polynomial equations
whose degrees are smaller than the elliptic period $m$.
We will find the following remarkable algebraic relations between caustic and
ellipsoidal parameters using the new formulation.
The billiard trajectories inside the ellipsoid $Q$ sharing the caustics
$Q_{\lambda_1}$ and $Q_{\lambda_2}$ are periodic with:
\begin{itemize}
\item
Elliptic period $m=3$ if the roots of $t^3-(t-c)(t-b)(t-a)$
are the caustic parameters;
\item
Elliptic period $m = 4$ if there exists $d \in \Rset$ such that the roots of
$t^4 - (t-d)^2(t-\lambda_1)(t-\lambda_2)$ are the ellipsoidal parameters; and
\item
Elliptic period $m=5$ if the roots of
$t^5-(t-c)(t-b)(t-a)(t-\lambda_1)(t-\lambda_2)$ are double ones.
\end{itemize}

Let us compare both formulations in the third case.
On the one hand,
the two homogeneous symmetric polynomial equations obtained from
the matrix formulation have degrees 23 and 24 in the variables
$1/a$, $1/b$, $1/c$, $1/\lambda_1$, and $1/\lambda_2$.
On the other hand, $t^5-(t-c)(t-b)(t-a)(t-\lambda_1)(t-\lambda_2)$ is a polynomial
of degree four in a single variable.
Clearly, the polynomial formulation leads to a much simpler problem.
Nevertheless, it should be stressed that the matrix formulation determines
\emph{all} periodic billiard trajectories with elliptic period $m=5$.
On the contrary, we find just a \emph{subset} of such trajectories using
the polynomial $t^5-(t-c)(t-b)(t-a)(t-\lambda_1)(t-\lambda_2)$.

Another natural question about periodic billiard trajectories
is the following one.
Which are the triaxial ellipsoids of $\Rset^3$ that display periodic billiard
trajectories with a fixed caustic type and a fixed (elliptic) period?
A numerical approach to that question was considered in~\cite{CasasRamirez2011},
where the authors computed several bifurcations in the space of ellipsoidal
parameters.
We will find the algebraic relations that define the bifurcations
associated to small elliptic periods.
For instance, we will see that there exist periodic billiard trajectories
with elliptic period $m = 3$ and caustic type EH1 if and only if
$c < ab/(a+b+\sqrt{ab})$.

For brevity,
we will not depict billiard trajectories
inside triaxial ellipsoids of $\Rset^3$.
The reader interested in 3D graphical visualizations is referred
to~\cite{CasasRamirez2012}, where several periodic billiard trajectories
with small periods are displayed from different perspectives.

We complete this introduction with a note on the organization of the article.
In Section~\ref{sec:Preliminaries} we review briefly some well-known results
about billiards inside ellipsoids, recalling the matrix formulation of
the generalized Cayley conditions obtained by Dragovi\'{c} and Radnovi\'{c}.
We also introduce the concept of elliptic period.
The practical limitations of the matrix formulation are exposed
in Section~\ref{sec:Practical}.
Next, we present the polynomial formulation in Section~\ref{sec:Polynomial}.
In Section~\ref{sec:Minimal} we carry out a detailed analysis for minimal
elliptic periods,
whereas the study of more general elliptic periods is postponed to
Section~\ref{sec:General}.
The previous results are adapted to billiards inside ellipses of $\Rset^2$
and inside triaxial ellipsoids of $\Rset^3$ in
sections~\ref{sec:Planar} and~\ref{sec:Spatial}, respectively.

\section{Preliminaries}
\label{sec:Preliminaries}

In this section we recall several classical results
and their modern generalizations about billiards inside ellipsoids
that go back to Jacobi, Chasles, Poncelet, Darboux, and Cayley.

We consider the billiard dynamics inside the ellipsoid
\begin{equation} \label{eq:Ellipsoid}
Q =
\left\{
x = (x_1,\ldots,x_n) \in \Rset^n : \sum_{i=1}^n \frac{x_i^2}{a_i} = 1
\right\},\quad
0 < a_1 < \cdots < a_n.
\end{equation}
The degenerate cases in which the ellipsoid has some symmetry of revolution
are not considered here.
This ellipsoid is an element of the family of confocal quadrics
\[
Q_{\lambda} =
\left\{
x = (x_1,\ldots,x_n) \in \Rset^n : \sum_{i=1}^n \frac{x_i^2}{a_i - \lambda} = 1
\right\}, \qquad \lambda \in \Rset.
\]
We note that $Q_\lambda = \emptyset$ for $\lambda > a_n$.
Thus, there are exactly $n$ different geometric types of nonsingular quadrics
in this family, which correspond to the cases
\[
\lambda \in (-\infty,a_1),\quad
\lambda \in (a_1,a_2),\quad \ldots, \quad
\lambda \in (a_{n-1},a_n).
\]
For instance, the confocal quadric $Q_\lambda$ is an ellipsoid
if and only if $\lambda \in (-\infty,a_1)$.
On the other hand, the meaning of $Q_\lambda$ in the singular cases
$\lambda \in \{ a_1, \ldots, a_n\}$ is
\[
Q_{a_j} = H_j =
\left\{ x=(x_1,\ldots,x_n) \in \Rset^n : x_j = 0 \right\}.
\]

The following theorems of Jacobi and Chasles can be found
in~\cite{KozlovTreschev1991,Tabachnikov1995,Tabachnikov2005}.

\begin{thm}[Jacobi]
\label{thm:Jacobi}
Any generic point $x \in \Rset^n$ belongs to exactly $n$
distinct nonsingular quadrics $Q_{\mu_0},\ldots,Q_{\mu_{n-1}}$ such
that $\mu_0 < a_1 < \mu_1 < a_2 < \cdots < a_{n-1} < \mu_{n-1} < a_n$.
\end{thm}

We denote by $\mu = \left( \mu_0, \ldots, \mu_{n-1} \right) \in \Rset^n$
the \emph{Jacobi elliptic coordinates} of the point
$x = \left( x_1, \ldots, x_n \right)$.
Cartesian and elliptic coordinates are linked by relations
\[
x^2_j =
\frac{\prod_{i = 0}^{n-1} ( a_j - \mu_i)}{\prod_{i \neq j} ( a_j - a_i)},
\qquad j = 1, \ldots, n.
\]
Hence, a point has the same elliptic coordinates that its
orthogonal reflections with respect to the coordinate subspaces of $\Rset^n$.
A point is \emph{generic}, in the sense of Theorem~\ref{thm:Jacobi},
if and only if it is outside all coordinate hyperplanes.
When a point tends to the coordinate hyperplane $H_j$,
some of its elliptic coordinates tends to $a_j$.

\begin{thm}[Chasles]
\label{thm:Chasles}
Any line in $\Rset^n$ is tangent to exactly $n-1$ confocal quadrics
$Q_{\lambda_1},\ldots,Q_{\lambda_{n-1}}$.
\end{thm}

It is known that if two lines obey the reflection law at a point $x \in Q$,
then both lines are tangent to the same confocal quadrics.
Thus, all lines of a billiard trajectory inside the ellipsoid $Q$ are
tangent to exactly $n-1$ confocal quadrics $Q_{\lambda_1},\ldots,Q_{\lambda_{n-1}}$,
which are called \emph{caustics} of the trajectory,
whereas $\lambda_1,\ldots,\lambda_{n-1}$ are the \emph{caustic parameters}
of the trajectory.
We will say that a billiard trajectory inside $Q$
is \emph{nonsingular} when it has $n-1$ distinct nonsingular caustics.
We mostly deal with nonsingular billiard trajectories in this paper.

The caustic parameters cannot take arbitrary values.
For instance, a line cannot be tangent to two different confocal ellipsoids,
and all caustics parameters must be positive.
The following complete characterization was given in~\cite{Knorrer1980,Audin1994}.

\begin{pro}\label{pro:Existenceconditions}
Let $\lambda_1 < \cdots < \lambda_{n-1}$ be some real numbers such that
\[
\{ a_1,\ldots,a_n \} \cap \{ \lambda_1,\ldots,\lambda_{n-1} \} = \emptyset.
\]
Set $a_0 = 0$.
Then there exist nonsingular billiard trajectories inside the ellipsoid $Q$
sharing the caustics $Q_{\lambda_1},\ldots,Q_{\lambda_{n-1}}$ if and only if
\begin{equation}\label{eq:ExistenceCondition}
\lambda_k \in (a_{k-1}, a_k) \cup (a_k, a_{k+1}), \qquad 
k = 1,\ldots,n-1.
\end{equation}
\end{pro}

\begin{defi}\label{defi:CausticType}
The \emph{caustic type} of a nonsingular trajectory is the vector
$\varsigma = (\varsigma_1,\ldots,\varsigma_{n-1}) \in \Zset^{n-1}$ such that
\[
\lambda_k \in (a_{\varsigma_k}, a_{\varsigma_k + 1}), \qquad
k = 1,\ldots,n-1.
\]
\end{defi}

We know from Proposition~\ref{pro:Existenceconditions}
that $\varsigma_k \in \{k-1,k\}$ for $k = 1,\ldots,n-1$.
Hence, there are exactly $2^{n-1}$ different caustic types.
The two caustic types in the planar case correspond to
ellipses: $\varsigma_1 = 0$, and hyperbolas: $\varsigma_1 = 1$.
The four caustic types in the spatial case correspond to
EH1: $\varsigma = (0,1)$, H1H1: $\varsigma = (1,1)$,
EH2: $\varsigma = (0,2)$, and H1H2: $\varsigma = (1,2)$.
The notations EH1, H1H1, EH2, and H1H2 were described in
the introduction.

Next, we recall a result about periodic billiard trajectories
inside ellipsoids.

\begin{thm}[Generalized Poncelet theorem]
\label{thm:Poncelet}
If a nonsingular billiard trajectory is closed after $m_0$ bounces
and has length $L_0$,
then all trajectories sharing the same caustics are also closed
after $m_0$ bounces and have length $L_0$.
\end{thm}

Poncelet proved this theorem for conics~\cite{Poncelet1822}.
Darboux generalized it to triaxial ellipsoids of $\Rset^3$; see~\cite{Darboux1887}.
Later on, this result was generalized to any dimension
in~\cite{ChangFriedberg1988,Chang_etal1993a,Chang_etal1993b,Previato1999}.

The periodic billiard trajectories sharing the same caustics
also have the same winding numbers.
In order to introduce these numbers, we set
\begin{equation}\label{eq:c}
\left\{ c_1,\ldots,c_{2n-1} \right\} =
\left\{ a_1,\ldots,a_n \right\} \cup
\left\{ \lambda_1,\ldots,\lambda_{n-1} \right\},
\end{equation}
and $f(t) = \sqrt{\prod_{i=1}^{2n-1}(1-t/c_i)}$.
We deal with nonsingular billiard trajectories inside
ellipsoids without symmetries of revolution,
so the parameters $c_1, \ldots, c_{2n-1}$ are pairwise distinct,
and we can assume that $c_0 := 0 < c_1 < \cdots < c_{2n-1}$.

\begin{thm}[Winding numbers]
\label{thm:WindingNumbers}
The nonsingular billiard trajectories inside the ellipsoid $Q$
sharing the caustics $Q_{\lambda_1},\ldots,Q_{\lambda_{n-1}}$
are periodic with period $m_0$ if and only if there exist some 
positive integer numbers $m_1,\ldots,m_{n-1}$ such that
\begin{equation}\label{eq:WindingNumbersIntegrals}
\sum_{j=0}^{n-1} (-1)^j m_j
\int_{c_{2j}}^{c_{2j+1}} \frac{t^i}{f(t)} \rmd t = 0,
\qquad \forall i=0,\ldots,n-2.
\end{equation}
Each of these periodic billiard trajectories has $m_j$ points at $Q_{c_{2j}}$
and $m_j$ points at $Q_{c_{2j+1}}$.
Besides,
$\{c_{2j},c_{2j+1}\} \cap \{a_1,\ldots,a_n\} \neq \emptyset
 \Rightarrow m_j$ even.
Finally, $\gcd(m_0,\ldots,m_{n-1}) \in \{1,2\}$.

Let $L_0$ be the common length of these periodic billiard trajectories.
Let $x(t)$ be an arc-length parametrization of any of these trajectories.
Let $\mu(t) = (\mu_0(t), \ldots, \mu_{n-1}(t))$ be the
corresponding parametrization in elliptic coordinates.
Then:
\begin{enumerate}
\item
$c_{2j} \leq \mu_j(t) \leq c_{2j+1}$ for all $t \in \Rset$.

\item
Functions $\mu_j(t)$ are smooth everywhere,
except $\mu_0(t)$ which is non-smooth at impact points
---that is, when $x(t_{\star}) \in Q$---,
in which case $\mu'_0 (t_{\star}+) = - \mu'_0(t_{\star}-) \neq 0$.
    
\item
If $\mu_j(t)$ is smooth at $t = t_{\star}$, then
$\mu_j'(t_{\star}) = 0 \Leftrightarrow
 \mu_j(t_{\star}) \in \{ c_{2j}, c_{2j+1} \}$.

\item
$\mu_j(t)$ makes exactly $m_j$ complete oscillations (round trips)
inside the interval $[c_{2j}, c_{2j+1}]$ along one period $0 \leq t \leq L_0$.

\item
$\mu(t)$ has period $L = L_0 / \gcd(m_0,\ldots,m_{n-1})$.
\end{enumerate}
\end{thm}

\begin{defi}
The numbers $m_0,\ldots,m_{n-1}$ are called \emph{winding numbers}.
Theorem~\ref{thm:WindingNumbers} contains three equivalent definitions for them:
by means of the property regarding hyperelliptic integrals given
in~(\ref{eq:WindingNumbersIntegrals}),
as a geometric description of how the periodic billiard trajectories
fold in $\Rset^n$,
and as the number of oscillations of the elliptic coordinates
along one period.
\end{defi}

Most of the statements of Theorem~\ref{thm:WindingNumbers}
can be found in~\cite{DragovicRadnovic2006,DragovicRadnovic2008},
but the one about the even character of some winding numbers and
the ones regarding $\gcd(m_0,\ldots,m_{n-1})$.
The first statement is trivial; suffice it to realize that
a periodic billiard trajectory can only have an even number of
crossings with any coordinate hyperplane.
The second ones follow from the oscillating behaviour of elliptic coordinates
along billiard trajectories described in Theorem~\ref{thm:WindingNumbers};
suffice it to note that all elliptic coordinates make an integer number
of complete oscillations inside their corresponding intervals along one
half-period $L_0/2$ when $\gcd(m_0,\ldots,m_{n-1}) = 2$.

The following conjecture was stated in~\cite{CasasRamirez2011},
where it was numerically tested.

\begin{con}\label{con:WindingNumbers}
Winding numbers are always ordered in a strict decreasing way;
that is,
\[
2 \leq m_{n-1} < \cdots < m_1 < m_0.
\]
\end{con}

It is known that the conjecture holds in the planar case.
If this conjecture holds, then any nonsingular periodic billiard trajectory
inside $Q$ has period at least $n+1$.
By the way, there are periodic billiard trajectories of smaller periods,
but all of them are singular
---they are contained in some coordinate hyperplane or in some
ruled quadric of the confocal family.

In light of the last item of Theorem~\ref{thm:WindingNumbers},
we present the following definitions.

\begin{defi}\label{defi:EllipticPeriod}
The \emph{elliptic period} $m$ and the \emph{elliptic winding numbers}
$\widetilde{m}_0,\ldots,\widetilde{m}_{n-1}$
of the nonsingular periodic billiard trajectories with period $m_0$
and winding numbers $m_0,\ldots,m_{n-1}$ are
\[
m = m_0/d, \qquad  
\widetilde{m}_j = m_j/d,
\]
where $d = \gcd(m_0,\ldots,m_{n-1})$.
\end{defi}

Roughly speaking,
the difference between the period $m_0$ and the elliptic period $m$
is that periodic billiard trajectories close in Cartesian (respectively,
elliptic) coordinates after exactly $m_0$ (respectively, $m$) bounces.
In order to clarify this difference,
let us consider the six planar periodic trajectories shown
in Figure~\ref{fig:Planar}; see Section~\ref{sec:Planar}.
Only the trajectory in Figure~\ref{fig:Planar:c} verifies that $m = m_0$.
On the contrary, the trajectories in figures~\ref{fig:Planar:a},
\ref{fig:Planar:b}, and~\ref{fig:Planar:e}
(respectively, Figure~\ref{fig:Planar:d})
(respectively,  Figure~\ref{fig:Planar:f}) have even period $m_0$
and any of their impact points becomes its reflection with respect to
the origin (respectively, the vertical axis)
(respectively, the horizontal axis) after $m_0/2$ bounces,
so they have elliptic period $m = m_0/2$.

It turns out that given any ellipsoid of the form~(\ref{eq:Ellipsoid})
and any proper coordinate subspace of $\Rset^n$,
there exist infinitely many sets of $n-1$ distinct nonsingular caustics
such that their tangent trajectories are periodic with even period,
say $m_0$, and any of their impact points becomes its reflection
with respect to that coordinate subspace after $m_0/2$ bounces.
We will not prove this claim,
since the proof requires some convoluted ideas
developed in~\cite{CasasRamirez2011,CasasRamirez2012}.

It is natural to look for caustics giving rise to
periodic billiard trajectories inside that ellipsoid.
Such caustics can be found by means of certain algebraic conditions,
called \emph{generalized Cayley conditions}.
They are found by working in elliptic coordinates,
so they depend on the elliptic period $m$, not on the (Cartesian) period $m_0$.

\begin{thm}[Generalized Cayley conditions]
\label{thm:Cayley}
The nonsingular billiard trajectories inside the ellipsoid $Q$
sharing the caustics $Q_{\lambda_1},\ldots,Q_{\lambda_{n-1}}$
are periodic with elliptic period $m$ if and only if
$m \ge n$ and
\[
\Rank
\left(
\begin{array}{ccc}
f_{m + 1}  & \cdots & f_{n + 1} \\
\vdots   &        & \vdots \\
f_{2m - 1} & \cdots & f_{m + n -1}
\end{array}
\right) < m - n + 1,
\]
where $f(t) = \sum_{l \ge 0} f_l t^l := \sqrt{\prod_{i=1}^{2n-1}(1-t/c_i)}$.
\end{thm}

Cayley proved this theorem for conics~\cite{Cayley1854}.
Later on, this result was generalized to any dimension by Dragovi\'{c}
and Radnovi\'{c} in~\cite{DragovicRadnovic1998a,DragovicRadnovic1998b}.
These authors have also given similar Cayley conditions in many other billiard
frameworks; see~\cite{DragovicRadnovic2004,DragovicRadnovic2006,DragovicRadnovic2006b,
DragovicRadnovic2008,DragovicRadnovic2012}.

\begin{defi}\label{defi:C_m_n}
$\mathcal{C}(m,n)$ denotes the generalized Cayley condition that characterizes
billiard trajectories of elliptic period $m$ inside ellipsoids of $\Rset^n$
given in Theorem~\ref{thm:Cayley}.
\end{defi}

\section{On the matrix formulation of the generalized Cayley conditions}
\label{sec:Practical}

The matrix formulation of the generalized Cayley condition $\mathcal{C}(m,n)$
stated in Theorem~\ref{thm:Cayley} is very nice from a theoretical point of view,
but has strong limitations from a practical point of view.
Let us describe them.

The function $f(t)$ is symmetric in the inverse quantities $\gamma_i = 1/c_i$.
In order to exploit it, we introduce some notations about symmetric polynomials.
Let $\Qset^\HS_l[x_1,\ldots,x_s]$ be the vectorial space over $\Qset$ of
all homogeneous symmetric polynomials with rational coefficients of degree $l$
in the variables $x_1,\ldots,x_s$.
Let $\rme_l(x_1,\ldots,x_s)$ be the \emph{elementary symmetric polynomial}
of degree $l$ in the variables $x_1,\ldots,x_s$. That is,
$\prod_{i=1}^s(1+x_i t) = \sum_{l \ge 0} \rme_l(x_1,\ldots,x_s) t^l$,
so $\rme_l(x_1,\ldots,x_s) = 0$ for all $l > s$.
Clearly, $\rme_l = \rme_l(x_1,\ldots,x_s) \in \Qset^\HS_l[x_1,\ldots,x_s]$. 

We stress that $f_l = f_l(\gamma_1,\ldots,\gamma_{2n-1}) \in
\Qset^\HS_l[\gamma_1,\ldots,\gamma_{2n-1}]$,
which is one of the reasons for the introduction of the
inverse quantities $\gamma_i = 1/c_i$.
Indeed, using that $f^2(t) = \prod_{i=1}^{2n-1}(1-\gamma_i t)$,
we get the recursive relations
\[
f_0 = 1,\qquad
2 f_l =
(-1)^l \rme_l(\gamma_1,\ldots,\gamma_{2n-1}) -
\sum_{k=1}^{l-1} f_k f_{l-k}, \quad
\forall l \ge 1.
\]
Hence, it is possible to compute recursively all Taylor coefficients $f_l$,
although their expressions are rather complicated when $l$ is big.
Nevertheless, the computation of the Taylor coefficients
$f_{n+1},\ldots,f_{2m-1}$ is the simplest step in the practical
implementation of the generalized Cayley condition $\mathcal{C}(m,n)$.
Next, we must impose that all $(m-n+1) \times (m-n+1)$ minors
of the matrix that appear in Theorem~\ref{thm:Cayley} vanish.
For simplicity, let us consider the minors formed by
the first $m-n$ rows and the $(m-n+l)$-th row of that matrix, 
for $l = 1,\ldots, n-1$.
Then the Cayley condition $\mathcal{C}(m,n)$ can be written
as the system of $n-1$ polynomial equations
\begin{equation}\label{eq:PolynomialSystem}
M_{m,n,l} = M_{m,n,l}(\gamma_1,\ldots,\gamma_{2n-1}) :=
\left |
\begin{array}{ccc}
f_{m + 1}  & \cdots & f_{n + 1} \\
\vdots   &        & \vdots \\
f_{2m - n} & \cdots & f_m \\
f_{2m - n + l} & \cdots & f_{m + l}
\end{array}
\right| = 0,\quad 1 \le l \le n-1.
\end{equation}
One can check that $M_{m,n,l}(\gamma_1,\ldots,\gamma_{2n-1})
 \in \Qset^\HS_{(m-n+2)m-n+l} [\gamma_1,\ldots,\gamma_{2n-1}]$
from the Leibniz formula for determinants.
This implies that the resolution of system~(\ref{eq:PolynomialSystem})
is a formidable challenge,
even from a purely numerical point of view and for
relatively small values of $m$.

We want to write down the solutions of system~(\ref{eq:PolynomialSystem})
in an explicit algebraic way.
Let us focus on the planar case $n = 2$,
when condition $\mathcal{C}(m,2)$ becomes a single homogeneous symmetric
polynomial equation of degree $m^2-1$ in three unknowns; namely,
\[
M_m = M_m(\gamma_1,\gamma_2,\gamma_3) :=
\left |
\begin{array}{ccc}
f_{m + 1}  & \cdots & f_3 \\
\vdots   &        & \vdots \\
f_{2m - 1} & \cdots & f_{m + 1}
\end{array}
\right| = 0.
\]
For instance, condition $\mathcal{C}(2,2)$ can be easily solved, since
\begin{eqnarray*}
-16 M_2 & = &
\gamma^3_1 + \gamma^3_2 + \gamma^3_3 - \gamma^2_1 \gamma_2 -
\gamma^2_1 \gamma_3 - \gamma^2_2 \gamma_1 - \gamma^2_2 \gamma_3 -
\gamma^2_3 \gamma_1 - \gamma^2_3 \gamma_2 + 2 \gamma_1 \gamma_2 \gamma_3  \\
& = &
(\gamma_1 - \gamma_2 - \gamma_3)
(\gamma_3 - \gamma_1 - \gamma_2)
(\gamma_2 - \gamma_3 - \gamma_1).
\end{eqnarray*}
The inverse quantities $\gamma_i = 1/c_i$
verify that $0 < \gamma_3 < \gamma_2 < \gamma_1$,
since $0 < c_1 < c_2 < c_3$.
Therefore, only the first factor of the above formula provides
a feasible solution,
and so condition $\mathcal{C}(2,2)$ has a unique solution:
\[
\gamma_1 = \gamma_2 + \gamma_3.
\]

The computations for condition $\mathcal{C}(3,2)$ are much harder,
so we have implemented them using a computer algebra system.
We got the factorization $-16384 M_3 = q_0 q_1 q_2 q_3$, where
\begin{eqnarray*}
q_0 & = &
\gamma_1^2 + \gamma_2^2 + \gamma_3^2 -
2 \gamma_1 \gamma_2 - 2 \gamma_1 \gamma_3 - 2\gamma_2 \gamma_3,\\
q_k & = &
3 \gamma^2_k - 2(\gamma_i + \gamma_j) \gamma_k - (\gamma_i  - \gamma_j)^2,\qquad
\{i,j,k\} = \{1,2,3\}.
\end{eqnarray*}
The first factor $q_0$ can, in its turn, be factored as
\[
q_0 =
(\sqrt{\gamma_1} - \sqrt{\gamma_2} - \sqrt{\gamma_3})
(\sqrt{\gamma_1} + \sqrt{\gamma_2} - \sqrt{\gamma_3})
(\sqrt{\gamma_1} - \sqrt{\gamma_2} + \sqrt{\gamma_3})
(\sqrt{\gamma_1} + \sqrt{\gamma_2} + \sqrt{\gamma_3}).
\]
The factor $q_0$ provides a unique feasible solution:
$\sqrt{\gamma_1} = \sqrt{\gamma_2} + \sqrt{\gamma_3}$,
because $0 < \gamma_3 < \gamma_2 < \gamma_1$.
Next, we consider the factor $q_k$ as a second-order polynomial in the variable
$\gamma_k$ with coefficients in $\Zset^\Symmetric[\gamma_i,\gamma_j]$.
Then we get the solutions
\[
\gamma_k = \gamma^\pm_k (\gamma_i,\gamma_j) :=
\frac{\gamma_i + \gamma_j}{3} \pm
\frac{2}{3} \sqrt{\gamma^2_i + \gamma^2_j - \gamma_i \gamma_j}.
\]
It turns out that $\gamma^-_k \le 0 < \max(\gamma_i,\gamma_j) < \gamma^+_k$,
so only the factor $q_1$ gives a solution compatible with the ordering
$0 < \gamma_3 < \gamma_2 < \gamma_1$;
namely, $\gamma_1 = \gamma^+_1(\gamma_2,\gamma_3)$.
Hence, $\mathcal{C}(3,2)$ has only two solutions:
\begin{equation}\label{eq:C32}
\sqrt{\gamma_1} = \sqrt{\gamma_2} + \sqrt{\gamma_3} \qquad \mbox{and} \qquad
3\gamma_1 =
\gamma_2 + \gamma_3 + 2\sqrt{\gamma^2_2 + \gamma^2_3 - \gamma_2 \gamma_3}.
\end{equation}

We have tried to write down explicitly the solutions of
system~(\ref{eq:PolynomialSystem}) in other cases,
but we did not succed,
even after implementing the computations in a computer algebra system.
This shows the limitations of the matrix formulation.

\section{A polynomial formulation of the generalized Cayley conditions}
\label{sec:Polynomial}

\begin{thm}\label{thm:Polynomial}
Let $r(t) = \prod_{i=1}^{2n-1} (1-t/c_i)$ and $f(t) = \sqrt{r(t)}$.
The generalized Cayley condition $\mathcal{C}(m,n)$
is equivalent to any of the following two conditions:
\begin{enumerate}[(i)]
\item\label{item:Polynomial-first}
There exists a non-zero polynomial $s(t) \in \Rset_{m-n}[t]$ such that
\begin{equation}\label{eq:sf}
\left.\frac{\rmd^l}{\rmd t^l}\right|_{t=0}
\left\{ s(t) f(t) \right\} = 0,\qquad
l = m+1,\ldots,2m-1.
\end{equation}
\item\label{item:Polynomial-second}
There exist $\alpha \neq 0$, $s(t) \in \Rset_{m-n}[t]$,
and $q(t) \in \Rset_{m-1}[t]$ such that $s(0) = q(0) = 1$ and
\begin{equation}\label{eq:sq}
s^2(t) r(t) = (\alpha t^m + q(t)) q(t).
\end{equation}
\end{enumerate}
\end{thm}

\begin{proof}
We split the proof in three steps.

\textit{Step 1: $\mathcal{C}(m,n) \Leftrightarrow$ (\ref{item:Polynomial-first})}.
$\mathcal{C}(m,n)$ means that the $m-n+1$ columns of the matrix
given in Theorem~\ref{thm:Cayley} are linearly dependent,
so there exist $s_0,\ldots,s_{m-n} \in \Rset$, not all zero, such that
\[
s_0 \times (\mbox{first column}) + \cdots +
s_{m-n} \times (\mbox{last column}) = 0,
\]
which is equivalent to condition~(\ref{eq:sf}) when
$s(t) = \sum_{l=0}^{m-n} s_l t^l \in \Rset_{m-n}[t]$.

\textit{Step 2: (\ref{item:Polynomial-first}) $\Rightarrow$ (\ref{item:Polynomial-second})}.
If $s(t) \in \Rset_{m-n}[t]$ verifies~(\ref{eq:sf}),
then $g(t) = \sum_{l \ge 0} g_l t^l := s(t) f(t)$
verifies that $g_l = 0$ for $l = m+1,\ldots,2m-1$.
Hence,
\[
g(t) = q(t) + \alpha t^m/2 + \Order(t^{2m}),
\]
where $q(t) = g_0 + \cdots + g_{m-1} t^{m-1} \in \Rset_{m-1}[t]$
and $\alpha = 2 g_m$.
Therefore,
\[
s^2 r = s^2 f^2 = g^2 = q^2 + \alpha t^m q + \Order(t^{2m}) =
(q + \alpha t^m) q + \Order(t^{2m}),
\]
and so $s^2 r = (q + \alpha t^m) q$,
since $\deg[s^2 r] \le 2m-1$ and $\deg[(q + \alpha t^m) q] \le 2m-1$.
Besides, $\alpha \neq 0$, because $\deg[r]$ is odd and $s(t) \not \equiv 0$.

Let $h(t) = \sum_{l \ge 0} h_l t^l := g^2(t) =
s^2(t) r(t) \in \Rset_{2m-1}[t]$.
Then
\[
0 = h_{2m} = \sum_{l=0}^{2m} g_l g_{2m-l} = (g_m)^2 + 2 g_0 g_{2m}
\Rightarrow g_0 g_{2m} = -\alpha^2/8 \neq 0.
\]
From this property, we deduce that
$q(0) = g_0 \neq 0$ and $s^2(0) = q^2(0)/r(0) \neq 0$.
Thus, we can normalize $s(t)$ by imposing $s(0) = 1$,
since condition~(\ref{eq:sf}) only determines $s(t)$
up to a multiplicative constant.
This implies that $q^2(0) = s^2(0) r(0) = 1$, so $q(0) = \pm 1$.
We can assume, without loss of generality, that $q(0) = 1$.
On the contrary,
we substitute $q(t)$ and $\alpha$ in the identity $s^2 r = (\alpha t^m + q) q$,
by $-q(t)$ and $-\alpha$, respectively.

\textit{Step 3: (\ref{item:Polynomial-second}) $\Rightarrow$ (\ref{item:Polynomial-first})}.
If there exist $\alpha \neq 0$,
$s(t) \in \Rset_{m-n}[t]$, and $q(t) \in \Rset_{m-1}[t]$
such that $s(0) = q(0) = 1$ and relation~(\ref{eq:sq}) holds,
we set $g(t) = \sum_{l \ge 0} g_l t^l := s(t) f(t)$.
Then,
\[
g^2 = s^2 f^2 = s^2 r = (q + \alpha t^m) q = (1+ \alpha t^m/q) q^2.
\]
The last operation is well-defined for small values of $|t|$,
because $q(0) \neq 0$.
Hence,
\[
g = \pm q \sqrt{1 + \frac{\alpha t^m}{q}} =
\pm q \left( 1 + \frac{\alpha t^m}{2q} + \Order(t^{2m}) \right) =
\pm q \pm \frac{\alpha t^m}{2} + \Order(t^{2m}),
\]
so $g_l = 0$ for $l = m+1,\ldots,2m-1$.
\end{proof}

Next, we present three examples of the results
that can be obtained from this formulation.

\begin{thm}\label{thm:Samples}
The nonsingular billiard trajectories inside the ellipsoid~(\ref{eq:Ellipsoid})
sharing the caustics $Q_{\lambda_1},\ldots,Q_{\lambda_{n-1}}$ are periodic with:
\begin{itemize}
\item
Elliptic period $m = n$ if the roots of $t^n - \prod_{j=1}^n (t-a_j)$ are
the caustic parameters;
\item
Elliptic period $m = n+1$ if there exists $d \in \Rset$ such that
the roots of $t^{n+1} - (t-d)^2 \prod_{k=1}^{n-1} (t-\lambda_k)$ are
the ellipsoidal parameters; and
\item
Elliptic period $m = 2n-1$ if all the roots of
$t^{2n-1} - \prod_{j=1}^n (t-a_j) \prod_{k=1}^{n-1}(t-\lambda_k)$ are double.
\end{itemize}
\end{thm}

\begin{proof}
Suffice it to find $\alpha \neq 0$, $s(t) \in \Rset_{m-n}[t]$,
and $q(t) \in \Rset_{m-1}[t]$ such that
\[
s^2(t) r(t) = (\alpha t^m + q(t)) q(t),\qquad
s(0) = q(0) = 1,
\]
for $m=n$, $m=n+1$, and $m=2n-1$, respectively;
see Theorem~\ref{thm:Polynomial}.
We recall that $r(t) = \prod_{i=1}^{2n-1} (1-t/c_i)$ with
$\{c_1,\ldots,c_{2n-1}\} =
 \{ a_1,\ldots,a_n \} \cup \{ \lambda_1,\ldots,\lambda_{n-1}\}$.

\textit{Case $m=n$}.
If the caustic parameters are the roots of $t^n - \prod_{j=1}^n (t-a_j)$,
then
\[
\textstyle
t^n - \prod_{j=1}^n (t-a_j) = \kappa \prod_{k=1}^{n-1} (t-\lambda_k),
\]
for some factor $\kappa \in \Rset$.
Indeed, $\kappa = \prod_{j=1}^n a_j \prod_{k=1}^{n-1} \lambda^{-1}_k$.
We take $\alpha = (-1)^n \prod_{j=1}^n a^{-1}_j$, $s(t) \equiv 1$,
and $q(t) = \prod_{k=1}^{n-1}(1-t/\lambda_k)$.
Clearly, $\alpha \neq 0$, $s(t) \in \Rset_0[t]$, $q(t) \in \Rset_{n-1}[t]$,
and $s(0) = q(0) = 1$.
Besides,
\[
s^2(t) r(t) =
\textstyle \prod_{j=1}^n (1-t/a_j) \prod_{k=1}^{n-1}(1-t/\lambda_k) =
(\alpha t^n + q(t)) q(t),
\]
since $\prod_{j=1}^n (1-t/a_j) =
\alpha \prod_{j=1}^n (t - a_j) =
\alpha t^n - \alpha \kappa \prod_{k=1}^{n-1}(t - \lambda_k) =
\alpha t^n + q(t)$.

\textit{Case $m=n+1$}.
If $a_1,\ldots,a_n$ are the roots of
$t^{n+1} - (t-d)^2 \prod_{k=1}^{n-1} (t-\lambda_k)$,
then
\[
\textstyle t^{n+1} - (t-d)^2 \prod_{k=1}^{n-1} (t-\lambda_k) =
\kappa \prod_{j=1}^n (t-a_j),
\]
for some $\kappa \in \Rset$.
Indeed,
$\kappa = d^2 \prod_{k=1}^{n-1} \lambda_k \prod_{j=1}^n a^{-1}_j$.
We take $\alpha = (-1)^{n+1} d^{-2} \prod_{k=1}^{n-1} \lambda^{-1}_k$,
$s(t) = (1-t/d)$, and $q(t) = \prod_{j=1}^n (1-t/a_j)$.
Clearly, $\alpha \neq 0$, $s(t) \in \Rset_1[t]$, $q(t) \in \Rset_n[t]$,
and $s(0) = q(0) = 1$.
Besides,
\[
s^2(t) r(t) =
(1-t/d)^2 \prod_{j=1}^n (1-t/a_i) \prod_{k=1}^{n-1}(1-t/\lambda_k) =
(\alpha t^{n+1} + q(t)) q(t),
\]
since $(1-t/d)^2\prod_{k=1}^{n-1} (1-t/\lambda_k) =
\alpha (t-d)^2 \prod_{k=1}^{n-1} (1-t/\lambda_k) =
\alpha t^{n+1} - \alpha \kappa \prod_{j=1}^n(t - a_j) =
\alpha t^{n+1} + q(t)$.

\textit{Case $m=2n-1$}.
If $t^{2n-1} -\prod_{i=1}^{2n-1} (t-c_i)$ has double roots $d_1,\ldots,d_{n-1}$,
then
\[
\textstyle
t^{2n-1} - \prod_{i=1}^{2n-1} (t-c_i) =
\kappa \prod_{l=1}^{n-1} (t-d_l)^2,
\]
for some $\kappa \in \Rset$.
Indeed, $\kappa = \prod_{i=1}^{2n-1} c_i \prod_{l=1}^{n-1} d^{-2}_l$.
We take $\alpha = -\prod_{i=1}^{2n-1} c^{-1}_i$,
$s(t) = \prod_{l=1}^{n-1} (1-t/d_l)$, and $q(t) = s^2(t)$.
Clearly, $\alpha \neq 0$, $s(t) \in \Rset_{n-1}[t]$,
$q(t) \in \Rset_{2n-2}[t]$, and $s(0) = q(0) = 1$.
Besides,
\[
s^2(t) r(t) =
\textstyle \prod_{l=1}^{n-1}(1 - t/d_l)^2 \prod_{i=1}^{2n-1} (1 - t/c_i) =
(\alpha t^{2n-1} + q(t)) q(t),
\]
since $\prod_{i=1}^{2n-1} (1-t/c_i) =
\alpha \prod_{i=1}^{2n-1} (t - c_i) =
\alpha t^{2n-1} - \alpha \kappa \prod_{l=1}^{n-1}(t - d_l)^2 =
\alpha t^{2n-1} + q(t)$.
\end{proof}

Several questions arise about the periodic trajectories
found in the previous theorem.
Let us mention just three.
Which are their caustic types, their (Cartesian) periods,
and their winding numbers?
Inside what ellipsoids exist them?
Are there other nonsingular periodic billiard trajectories
with elliptic period three, four or five?

We will give some partial answers in the next sections.

Some technicalities become simpler after
the change of variables $t = 1/x$.
Thus, we state another polynomial formulation of
the generalized Cayley condition $\mathcal{C}(m,n)$.

\begin{pro}\label{pro:Polynomial}
Let $R(x) = x \prod_{i=1}^{2n-1} (x - \gamma_i)$, where $\gamma_i = 1/c_i$.
The generalized Cayley condition $\mathcal{C}(m,n)$ holds if and only if
there exist two monic polynomials $S(x),R(x) \in \Rset[x]$ such that
$\deg[S] = m-n$, $\deg[P] = m$, $P(0) \neq 0$, and
\begin{equation}\label{eq:SP}
S^2(x) R(x) = P(x) \big( P(x) - P(0) \big).
\end{equation}
Furthermore, if such polynomials $S(x)$ and $P(x)$ exist,
the following properties hold:
\begin{enumerate}
\item
$S(x)$ has no multiple roots;
\item
All the real roots of $S(x)$ are contained in $\{ x \in \Rset : R(x) < 0 \}$;
\item
All the roots of $S(x)$ are real when $m \le n+3$; and
\item
$P(x)$ and $P(x)-P(0)$ have the same number of real roots
(counted with multiplicity).
\end{enumerate}
\end{pro}

\begin{proof}
The ``if and only if'' follows directly from the change of variables $t = 1/x$.
Concretely, the relation between the objects of identities~(\ref{eq:sq})
and~(\ref{eq:SP}) is
\[
R(x) = x^{2n} r(1/x),\qquad
S(x) = x^{m-n} s(1/x),\qquad
P(x) = \alpha + x^m q(1/x).
\]
Then $P(0) \neq 0$ if and only if $\alpha \neq 0$,
$s(0) = 1$ if and only if $S(x)$ is a monic polynomial of degree $m-n$,
and $q(0) = 1$ if and only if $P(x)$ is a monic polynomial of degree $m$.

To prove the first two properties,
suffice it to prove that $\gcd[S,RS'] = 1$ and
\[
l_+ := \# \{ x \in \Rset: S(x) = 0 < R(x) \} = 0.
\]
If $W(x) = P(x) (P(x) - P(0))$ and $T(x) = P(x) - P(0)/2$,
we get from~(\ref{eq:SP}) that
\begin{eqnarray*}
W(x) & = & S^2(x)R(x) = T^2(x) - P^2(0)/4, \\
W'(x) & = & S(x) ( S(x) R'(x) + 2 R(x) S'(x) ) = 2 T(x) P'(x).
\end{eqnarray*}

We consider the factorization $W'(x) = 2m W_-(x) W_0(x) W_+(x) W_\ast(x)$,
where if $z \in \Cset$ is a root of multiplicity $\beta$ of $W'(x)$
such that $W(z) < 0$, $W(z) = 0$, $W(z) > 0$, or $W(z) \not\in \Rset$,
then $(x-z)^\beta$ is included in the monic factor
$W_-(x)$, $W_0(x)$, $W_+(x)$, or $W_\ast(x)$, respectively.
Next, we find some lower bounds of the degrees of these factors.

First, $T$ is divisor of $W_-$,
because $W$ takes the negative value $-P^2(0)/4$ at each root of $T$.
Hence, $\deg[W_-] \ge \deg[T] = m$.
Second, $S \gcd[S,RS']$ is a divisor of $W_0$,
because $W$ vanishes at each root of $S$.
Thus, $\deg[W_0] \ge m - n + l_0$,
where $l_0$ denotes the degree of $\gcd[S,RS']$.
Third,
\[
\deg[W_+] \ge
\# \left\{
(a,b) \subset \Rset:
\begin{array}{l}
W(a) = W(b) = 0 \\
\mbox{$R(x) > 0$ for all $x\in(a,b)$} \\
\mbox{$S(x) \neq 0$ for all $x \in (a,b)$} 
\end{array}
\right\} =
n-1 + l_+.
\]
To understand the above inequality,
we realize that if $(a,b)$ is an open bounded interval
that satisfies the above three properties,
then $W(x) = S^2(x)R(x) > 0$ for all $x \in (a,b)$,
and $W'(x)$ vanishes at some point $c \in (a,b)$, by Rolle's Theorem.
Therefore, $\deg[W_+]$ is at least the number of such intervals.
We combine these three lower bounds:
\[
2m-1 = \deg[W'] \ge \deg[W_-] + \deg[W_0] + \deg[W_+]
\ge 2m - 1 + l_0 + l_+.
\]
This implies that $l_0 = l_+ = 0$.
Indeed, $W_- = T$, $W_0 = S$, $W_\ast = 1$, and $\gcd[S,RS'] = 1$.

Next, we prove the property about the number of roots of
$P(x)$ and $P(x) - P(0)$.
Let $z$ be a root of the derivative $P'$.
Since $W' = 2 T P'$ and $W_- = T$,
we deduce that $W(z)$ cannot be a negative number.
This implies that if $P(z)$ is a real value between $0$ and $P(0)$,
then $P'(z) \neq 0$, since $W(z) = P(z) (P(z) - P(0)) < 0$.
In particular,
we deduce that the number of real roots (counted with multiplicity)
of the polynomial $P(x)-\eta$ does not change when
the constant $\eta \in \Rset$ moves from $0$ to $P(0)$.

Finally, we prove that $S(x)$ has only real roots when $m \le n+3$.
Let us suppose that $z \not \in \Rset$ is a root of $S(x)$.
Then $\bar{z}$ is also a root of $S(x)$,
so $\Divide{(x-z)(x-\bar{z})}{S(x)}$.
Using the identity $S^2(x)R(x) = P(x)(P(x)-P(0))$,
we get that $(x-z)^2(x-\bar{z})^2$ is either a divisor of $P(x)$
or a divisor of $P(x)-P(0)$,
since $P(x)$ and $P(x) - P(0)$ have no common factors.
But $P(x)$ and $P(x)-P(0)$ have the same number of real roots,
so there exists another $w \not \in \Rset \cup \{ z,\bar{z}\}$
such that $(x-w)^2(x-\bar{w})^2$ is a divisor of $P(x)(P(x)-P(0))$.
This implies that $S(x)$ has, at least, four different complex roots,
and so $m-n = \deg[S] \ge 4$. 
\end{proof}

There are some theoretical arguments against the existence of
non-real roots of polynomial $S(x)$,
although we have not been able to prove it.

\begin{con}\label{con:Roots_of_S}
Let $R(x) = x \prod_{i=1}^{2n-1} (x - \gamma_i)$ with
$0 < \gamma_{2n-1} < \cdots < \gamma_1$.
If relation~(\ref{eq:SP}) holds for some polynomials
$S(x),P(x) \in \Rset[x]$ such that $P(0) \neq 0$,
then $S(x)$ has only real roots.
\end{con}

\section{Generalized Cayley conditions in the minimal case}
\label{sec:Minimal}

Let us consider the case of minimal elliptic periods; that is, $m = n$.

We begin with a technical lemma to describe how the roots of the polynomials
of the form $P(x) (P(x)-P(0))$ with $P(0) \neq 0$ are ordered in the real line,
assuming that all these roots are positive ---but the trivial one---,
and have multiplicity at most two.

\begin{lem}\label{lem:Ordering}
Let $P(x) \in \Rset[x]$ be a monic polynomial of degree $m$ such that
$P(0) \neq 0$ and all the roots of $P(x)(P(x)-P(0))$ are positive
---but a simple root at $x=0$---, and have multiplicity at most two.
Let $\alpha_m \le \cdots \le \alpha_1$ be the positive roots of $P(x)$.
Let $\beta_{m-1} \le \cdots \le \beta_1$ be the positive roots of $P(x) -P(0)$.

If $m$ is odd,
then $\beta_{2l-1}, \beta_{2l} \in (\alpha_{2l},\alpha_{2l-1})$
for all $l = 1,\ldots,(m-1)/2$; so
\[
0 < \alpha_m \le \alpha_{m-1} < \beta_{m-1} \le \beta_{m-2} <
 \alpha_{m-2} \le \alpha_{m-3} < \cdots < \alpha_3 \le \alpha_2 < \beta_2 \le \beta_1 < \alpha_1.
\]
If $m$ is even, then $\beta_1 > \alpha_1$ and
$\beta_{2l},\beta_{2l+1} \in (\alpha_{2l+1},\alpha_{2l})$
for all $l = 1,\ldots,(m-2)/2$; so
\[
0 < \alpha_m \le \alpha_{m-1} < \beta_{m-1} \le \beta_{m-2} <
\alpha_{m-2} \le \alpha_{m-3} < \cdots <
\beta_3 \le \beta_2 < \alpha_2 \le \alpha_1 < \beta_1.
\]
\end{lem}

\begin{proof}
Let $\eta \in \Rset$.
Using that the only critical points of $P(x)$
are non-degenerate local maxima or non-degenerate local minima,
we deduce that the polynomial $P(x)-\eta$ has $m$ real roots
(counted with multiplicity) if and only if
$\underline{\eta} \le \eta \le \overline{\eta}$, where
\begin{eqnarray*}
\overline{\eta} & = &
\min
\left\{
P(\overline{x}) : \mbox{$\overline{x}$ is a non-degenerate local maximum of $P(x)$}
\right\}, \\
\underline{\eta} & = &
\max
\left\{
P(\underline{x}) : \mbox{$\underline{x}$ is a non-degenerate local minimum of $P(x)$}
\right\}.
\end{eqnarray*}
Therefore, $\underline{\eta} \le \min(0,P(0))$ and
$\overline{\eta} \ge \max(0,P(0))$.

We begin with the case $m$ odd,
so $P(0) = (-1)^m \prod_{j=1}^m \alpha_j < 0$,
$\underline{\eta} \le P(0)$, and $\overline{\eta} \ge 0$.
The roots of $P(x)$ and $P(x)-P(0)$ can be viewed as the abscissae of
the intersections of the graph $\{ y = P(x)\}$ with the horizontal
line $\{ y = 0 \}$ and $\{ y = P(0) \}$, respectively.
Double roots correspond to tangential intersections.
We know that $P(\overline{x}) \ge \overline{\eta} \ge 0$ at the local maxima,
and $P(\underline{x}) \le \underline{\eta} \le P(0)$ at the local minima.
This means that the intersections of the graph $\{ y = P(x) \}$ with
the lines $\{ y = 0 \}$ and $\{ y = P(0) \}$ have the following pattern
from left to right.
First, the graph crosses $\{ y = P(0) \}$ at the abscissa $x = 0$;
second, it intersects $\{ y = 0 \}$ at two abscissae $\alpha_m$ and $\alpha_{m-1}$,
which may coincide giving rise to a double root of $P(x)$;
third, it intersects $\{ y = P(0) \}$ at two abscissae $\beta_{m-1}$ and $\beta_{m-2}$,
which may coincide giving rise to a double root of $P(x) - P(0)$;
fourth, it intersects $\{ y = 0 \}$ at two abscissae $\alpha_{m-2}$ and $\alpha_{m-3}$,
which may coincide giving rise to a double root of $P(x)$;
and so on.
The last intersection correspond to the abscissa $x = \alpha_1$.

The proof for $m$ even is similar.
We skip the details.
\end{proof}

We emphasize that ellipsoidal parameters $0 < a_1 < \cdots < a_n$ and
nonsingular caustic parameters $\lambda_1 < \cdots < \lambda_{n-1}$
verify restrictions~(\ref{eq:ExistenceCondition});
then the parameters $0 < c_1 < \cdots < c_{2n-1}$ are defined in~(\ref{eq:c});
next the inverse quantities $0 < \gamma_{2n-1} < \cdots < \gamma_1$
are given by $\gamma_i = 1/c_i$;
and finally, $R(x) = x \prod_{i = 1}^{2n-1} (x - \gamma_i)$.
We will make use of these notations, orderings, and conventions along
the paper without any explicit mention.

\begin{cor}\label{cor:JK}
Let $\{1,\ldots,2n-1 \} = J_n \cup K_n$ be the decomposition defined by
\[
J_1 = \{ 1 \},\quad
J_2 = \{2,3\},\quad
J_n = J_{n-2} \cup \{2n-2,2n-1\},\quad
K_n = J_{n-1}.
\]
If $P(x) \in \Rset[x]$ is a monic polynomial of degree $n$
such that $P(0) \neq 0$ and
\[
R(x) = P(x) (P(x) - P(0)),
\]
then $P(x) = \prod_{j \in J_n} (x-\gamma_j) =
P(0) + x \prod_{k \in K_n} (x-\gamma_k)$.
\end{cor}

\begin{proof}
There exists a decomposition $\{1,\ldots,2n-1 \} = J' \cup K'$
such that $\# J' = n$, $\# K' = n-1$, and
$P(x) =
 \prod_{j \in J'} (x-\gamma_j) =
 P(0) + x \prod_{k \in K'} (x-\gamma_k)$.
The polynomial $P(x)$ verifies the hypotheses stated in Lemma~\ref{lem:Ordering},
so the roots $\{ \alpha_1, \ldots, \alpha_m \} = \{ \gamma_j : j \in J' \}$ and
$\{ \beta_1, \ldots, \beta_{m-1} \} = \{ \gamma_k : k \in K' \}$ obey
the ordering described in that lemma.
Therefore, $J' = J_n$ and $K' = K_n$.
\end{proof}

We rewrite now the generalized Cayley condition $\mathcal{C}(n,n)$
using the previous results.
For brevity, we omit the dependence of the decomposition
$\{1,\ldots,2n-1\} = J \cup K$ on the index $n$.
We note that $\# J = n$ and $\# K = n-1$.
The symbol $\rme_l(\mbox{``a set of parameters''})$ denotes the
\emph{elementary symmetric polynomial} of degree $l$ in
those parameters.

\begin{pro}\label{pro:C_n_n}
$\mathcal{C}(n,n)$ is equivalent to any of the following four conditions:
\begin{enumerate}[(i)]
\item\label{item:C_n_n-first}
If $P(x) = \prod_{j \in J} (x-\gamma_j)$,
then $P(x) - P(0) = x \prod_{k \in K} (x-\gamma_k)$.

\item\label{item:C_n_n-second}
$\prod_{j \in J} (\gamma_j - \gamma_k) = \prod_{j \in J} \gamma_j$, for all $k \in K$.

\item\label{item:C_n_n-third}
$\rme_l ( \{\gamma_j\}_{j \in J}) = \rme_l ( \{\gamma_k\}_{k \in K})$,
for all $l=1,\ldots,n-1$.

\item\label{item:C_n_n-fourth}
$\sum_{j \in J} \gamma_j^l = \sum_{k \in K} \gamma_k^l$,
for all $l=1,\ldots,n-1$.
\end{enumerate}
\end{pro}

\begin{proof}
We split the proof in four steps.

\textit{Step 1: $\mathcal{C}(n,n) \Leftrightarrow$ (\ref{item:C_n_n-first}).}
Let us assume that $\mathcal{C}(n,n)$ holds.
Then there exist a monic polynomial $P(x) \in \Rset[x]$
of degree $n$ such that $P(0) \neq 0$ and
\[
R(x) = P(x) ( P(x) - P(0) ).
\]
Thus, condition~(\ref{item:C_n_n-first}) follows from Corollary~\ref{cor:JK}.

Reciprocally, if condition~(\ref{item:C_n_n-first}) holds,
$P(x) ( P(x) -P(0) ) = x \prod_{i=1}^{2n-1} (x-\gamma_i)$,
so $\mathcal{C}(n,n)$ holds.

\textit{Step 2: (\ref{item:C_n_n-first}) $\Leftrightarrow$ (\ref{item:C_n_n-second}).}
If $P(x) = \prod_{j \in J} (x-\gamma_j)$ and
$Q(x) = x \prod_{k \in K} (x-\gamma_k)$, then
\[
Q(x) = P(x)-P(0) \Leftrightarrow
P(\gamma_k) = P(0), \quad \forall k \in K \Leftrightarrow
\textstyle \prod_{j \in J} (\gamma_j - \gamma_k) = \prod_{j \in J} \gamma_j,
\quad \forall k \in K.
\]

\textit{Step 3: (\ref{item:C_n_n-first}) $\Leftrightarrow$ (\ref{item:C_n_n-third}).}
If $P(x) = \prod_{j \in J} (x-\gamma_j)$ and
$Q(x) = x \prod_{k \in K} (x-\gamma_k)$, then
\[
Q(x) =
x^n + \sum_{l=1}^{n-1} (-1)^l \rme_l( \{\gamma_k\}_{k \in K} ) x^{n-l}, \qquad
P(x) =
x^n + \sum_{l=1}^{n-1} (-1)^l \rme_l( \{\gamma_j\}_{j \in J} ) x^{n-l} + P(0).
\]

\textit{Step 4: (\ref{item:C_n_n-third}) $\Leftrightarrow$ (\ref{item:C_n_n-fourth}).}
It follows from the Newton's identities connecting
the elementary symmetric polynomials
and the power sum symmetric polynomials; see~\cite{Mead1992}.
\end{proof}

\begin{example}\label{ex:Planar}
The quantities $\gamma_3 = 1$, $\gamma_2 = 2$, and $\gamma_1 = 3$
verify $\mathcal{C}(2,2)$, because $1 + 2 = 3$.
This means that the billiard trajectories
\begin{itemize}
\item
Inside the ellipse $Q : x^2 + 2 y^2 = 1$ with caustic parameter
$\lambda = 1/3$; or
\item
Inside the ellipse $Q : x^2 + 3 y^2 = 1$ with caustic parameter
$\lambda = 1/2$;
\end{itemize}
are periodic with elliptic periodic $m = 2$.
Their caustic types are $\varsigma = 0$ and $\varsigma = 1$, respectively.
\end{example}

\begin{example}\label{ex:Spatial}
The quantities $\gamma_5 = 1$, $\gamma_4 = 2$, $\gamma_3 = 4$,
$\gamma_2 = 5$, and $\gamma_1 = 6$ verify $\mathcal{C}(3,3)$,
because $1 + 2 + 6 = 4 + 5$ and $1^2 + 2^2 + 6^2 = 4^2 + 5^2$.
Hence, the billiard trajectories
\begin{itemize}
\item
Inside the ellipsoid $Q : x^2 + 2 y^2 + 5 z^2 = 1$ with caustic parameters
$\lambda_1 = \frac{1}{6}$ and $\lambda_2 = \frac{1}{4}$; or
\item
Inside the ellipsoid $Q : x^2 + 2 y^2 + 6 z^2 = 1$ with caustic parameters
$\lambda_1 = \frac{1}{5}$ and $\lambda_2 = \frac{1}{4}$; or
\item
Inside the ellipsoid $Q : x^2 + 4 y^2 + 5 z^2 = 1$ with caustic parameters
$\lambda_1 = \frac{1}{6}$ and $\lambda_2 = \frac{1}{2}$; or
\item
Inside the ellipsoid $Q : x^2 + 4 y^2 + 6 z^2 = 1$ with caustic parameters
$\lambda_1 = \frac{1}{5}$ and $\lambda_2 = \frac{1}{2}$;
\end{itemize}
are periodic with elliptic periodic $m = 3$.
Their caustic types are $\varsigma = (0,1)$, $\varsigma = (1,1)$,
$\varsigma = (0,2)$, and $\varsigma = (1,2)$, respectively.
We recall that these caustic types were denoted EH1, H1H1, EH2, and H1H2
in the introduction.
\end{example}

Let us compare the system of homogeneous symmetric polynomial
equations~(\ref{eq:PolynomialSystem}),
which was obtained directly from the matrix formulation,
with the system of homogeneous non-symmetric polynomial equations
$\sum_{j \in J} \gamma_j^l = \sum_{k \in K} \gamma_k^l$,
$1 \le l \le n-1$, obtained in the previous proposition.
We are dealing with the case $m = n$,
so the $l$-th equation of the former system has degree $n+l$,
whereas the $l$-th equation of the new system has degree $l$.
Besides, the new system has a remarkably simple closed expression.
This shows that the polynomial formulation
simplifies the problem.

The beauty of the conditions regarding the elementary symmetric polynomials
and the power sum symmetric polynomials given in Proposition~\ref{pro:C_n_n}
has been the motivation for the introduction of the inverse quantities
$\gamma_i = 1/c_i$.
Nevertheless, we find useful to state the following result
in terms of the ellipsoidal parameters $a_j$,
in order to answer some questions about the nonsingular periodic billiard
trajectories found in the first item of Theorem~\ref{thm:Samples}.

\begin{thm}\label{thm:Main}
There exist nonsingular periodic billiard trajectories inside
the ellipsoid~(\ref{eq:Ellipsoid}) with elliptic period $m = n$ and
caustic type
\begin{equation}\label{eq:CausticType}
\varsigma =
\left\{
\begin{array}{ll}
(1,1,3,3,\ldots,n-2,n-2)   & \mbox{for $n$ odd} \\
(0,2,2,4,4,\ldots,n-2,n-2) & \mbox{for $n$ even}
\end{array}
\right.
\end{equation}
if and only if all the roots of $t^n - \prod_{j=1}^n (t-a_j)$ are real
and simple.
These periodic billiard trajectories
have the roots of $t^n - \prod_{j=1}^n (t-a_j)$ as caustic parameters,
period $m_0 = 2n$, and even winding numbers $m_0,\ldots,m_{n-1}$.
Indeed,
\begin{equation}\label{eq:WindingNumbers_C_n_n}
m_j = 2 \widetilde{m}_j = 2(n-j), \qquad j = 0,\ldots,n-1,
\end{equation}
provided Conjecture~\ref{con:WindingNumbers} on the strict decreasing ordering
of winding numbers holds.
\end{thm}

\begin{proof}
Let us assume that there exist nonsingular periodic billiard trajectories
with elliptic period $n$ and caustic type~(\ref{eq:CausticType}).
By definition of caustic type,
the caustic parameters $\lambda_1 < \cdots < \lambda_{n-1}$ of such trajectories
verify that
\begin{itemize}
\item
If $n$ is odd, then $\lambda_{2l-1},\lambda_{2l} \in (a_{2l-1},a_{2l})$,
for $l = 1,\ldots,(n-1)/2$;
\item
If $n$ is even, then $\lambda_1 \in (0, a_1)$, and
$\lambda_{2l},\lambda_{2l+1} \in (a_{2l},a_{2l+1})$,
for $l = 1,\ldots,(n-2)/2$.
\end{itemize}
Hence, we can split the set
$\{ \gamma_i = 1/c_i : i = 1,\ldots,2n-1 \}$ as
the disjoint union of the sets
\[
\{ \gamma_j : j \in J \} = \{ 1/a_1,\ldots,1/a_n \},\qquad
\{ \gamma_k : k \in K \} = \{ 1/\lambda_1,\ldots,1/\lambda_{n-1} \},
\]
where $\{1,\ldots,2n-1 \} = J \cup K$ is the decomposition described
in Corollary~\ref{cor:JK}.
Thus, we know from condition~\ref{item:C_n_n-fourth}
of Proposition~\ref{pro:C_n_n} that
\[
\prod_{j=1}^n \left(\frac{1}{a_j}-\frac{1}{\lambda_k}\right) =
\prod_{j=1}^n \frac{1}{a_j},\qquad
k = 1,\ldots,n-1.
\]
This identity can be written as
$\lambda_k^n = \prod_{j=1}^n (\lambda_k - a_j)$,
for all $k = 1,\ldots,n-1$, which implies that the
caustic parameters $\lambda_1,\ldots,\lambda_{n-1}$ are the roots
of $t^n - \prod_{j=1}^n (t-a_j)$.

Reciprocally, let us assume that the roots of
$q(t) = t^n - \prod_{j=1}^n (t-a_j)$ are real and simple.
Let $\lambda_1 < \cdots < \lambda_{n-1}$ be these roots.
None of them is zero, since $q(0) \neq 0$.
Besides,
$\lambda_k$ is a root of $q(t)$ if and only if
$\beta_k := 1/\lambda_k \neq 0$ is a root of
\[
Q(x) := \frac{(-1)^{n-1} x^n q(1/x)}{a_1 \cdots a_n} = P(x) - P(0),
\]
where $P(x) = \prod_{j=1}^n (x - \alpha_j)$ with $\alpha_j = 1/a_j$.
Therefore, the roots $\alpha_j = 1/a_j$ and $\beta_k = 1/\lambda_k$
are ordered as stated in Lemma~\ref{lem:Ordering}.
The consequences are two-fold.
On the one hand,
$\lambda_k \in (a_{\varsigma_k}, a_{\varsigma_k + 1})$,
where $\varsigma = (\varsigma_1,\ldots,\varsigma_{n-1})$
is the caustic type given in~(\ref{eq:CausticType}).
On the other hand,
there exist nonsingular billiard trajectories inside the ellipsoid $Q$
sharing the caustics $Q_{\lambda_1},\ldots,Q_{\lambda_{n-1}}$,
because the existence conditions~(\ref{eq:ExistenceCondition}) hold.
Thus, the trajectories sharing the caustics
$Q_{\lambda_1},\ldots,Q_{\lambda_{n-1}}$ are periodic with elliptic period $n$
and caustic type $\varsigma$,
since the generalized Cayley condition $\mathcal{C}(n,n)$ holds;
see Proposition~\ref{pro:C_n_n}.

Next, we prove the claims on the (Cartesian) period and the winding numbers.
The caustic parameters are located in the intervals delimited by
the ellipsoidal parameters given at the beginning of the proof,
which implies that
\[
\{c_{2j},c_{2j+1}\} \cap \{a_1,\ldots,a_n\} \neq \emptyset,
\qquad j = 0,\ldots,n-1,
\]
where $c_0 := 0 < c_1 < \cdots < c_{2n-1}$ are defined in~(\ref{eq:c}).
Thus, all winding numbers are even ---see Theorem~\ref{thm:WindingNumbers}---
and so, by definition of elliptic period, $m_0 = 2m = 2n$.

Finally, let us assume that winding numbers are ordered
as stated in Conjecture~\ref{con:WindingNumbers},
so $2 \le m_{n-1} < \cdots < m_0 = 2n$
with $m_0,\ldots,m_{n-1}$ even.
Then $m_j = 2\widetilde{m}_j = 2(n-j)$.
\end{proof}

\begin{remark}
If all the roots of $t^n - \prod_{j=1}^n(t-a_j)$ are real,
but some of them are double,
then we get singular periodic billiard trajectories.
In that case, there are only two possible scenarios.
Either $n$ is odd and $\lambda_{2l-1} = \lambda_{2l}$
for some $l=1,\ldots,(n-1)/2$; or $n$ is even and
$\lambda_{2l} = \lambda_{2l+1}$ for some $l=1,\ldots,(n-2)/2$.
In all these cases,
the singular periodic trajectories are formed by segments contained
in some nonsingular ruled confocal quadrics.
\end{remark}

\begin{remark}
All the periodic billiard trajectories mentioned in Theorem~\ref{thm:Main}
have caustic type~(\ref{eq:CausticType}).
One may establish similar theorems for other caustic types.
For instance,
the versions EH1, H1H1, EH2, and H1H2 of Theorem~\ref{thm:Main}
in the spatial case will be listed in Table~\ref{tab:Spatial}.
\end{remark}

\section{Cayley conditions in the general case}
\label{sec:General}

Once we have understood the minimal case $m = n$,
we tackle the general case $m \ge n$.

Let us explain the fundamental question by means of an example.
In Section~\ref{sec:Practical} we saw that condition $\mathcal{C}(3,2)$
becomes a single homogeneous symmetric polynomial equation of degree eight
in the variables $\gamma_1,\gamma_2,\gamma_3$,
with only two feasible solutions; namely, the ones given in~(\ref{eq:C32}).
Thus, it is natural to ask whether can we rewrite $\mathcal{C}(3,2)$
as a set of two simpler conditions such that each one of them gives rise to
one of the solutions given in~(\ref{eq:C32}).

By the way, we raise this question for any $m \ge n$.
Can we rewrite $\mathcal{C}(m,n)$ as a set of ``simpler'' conditions such that
each one of them gives rise to just ``one'' solution of $\mathcal{C}(m,n)$?
We answer this question in the affirmative.
Indeed, we parametrize these ``simpler'' conditions by
the elements of the set
\[
\mathcal{T}(m,n) =
\big\{
(\tau_1,\ldots,\tau_n) \in \Zset^n :
\tau_1 + \cdots + \tau_n = m-n,\; \tau_1,\ldots,\tau_n \ge 0
\big\}.
\]
The cardinal of $\mathcal{T}(m,n)$ is the number of monomials
of degree $m-n$ in $n$ variables.
Thus, $\# \mathcal{T}(m,n) = {m-1 \choose n-1}$,
which gives a precise estimate of the complexity of the
Cayley condition $\mathcal{C}(m,n)$ when $m$ grows.
We will refer to the elements of $\mathcal{T}(m,n)$ as \emph{signatures}.
We set $\gamma_{2n} = 0$ in order to simplify some notations.

\begin{defi}\label{defi:C_m_n_tau}
Given any signature $\tau = (\tau_1,\ldots,\tau_n) \in \mathcal{T}(m,n)$,
we say that condition $\mathcal{C}(m,n;\tau)$ holds if and only if
there exist two monic polynomials $S(x),P(x) \in \Rset[x]$ such that
$\deg[S] = m-n$, $\deg[P] = m$, $P(0) \neq 0$,
$S^2(x) R(x) = P(x) ( P(x) - P(0))$, and
$S(x)$ has $m-n$ simple real roots $\delta_{m-n} < \cdots < \delta_1$ such that
\begin{equation}\label{eq:tau}
\# \big (
\{\delta_1,\ldots,\delta_{m-n} \} \cap (\gamma_{2r},\gamma_{2r-1})
\big ) = \tau_r,
\qquad r = 1,\ldots,n.
\end{equation}
\end{defi}

\begin{cor}\label{cor:C_m_n_tau}
If there exists $\tau  \in \mathcal{T}(m,n)$ such that
$\mathcal{C}(m,n;\tau)$ holds, then $\mathcal{C}(m,n)$ also holds.
The reciprocal implication is true for $m \le n+3$
(or provided Conjecture~\ref{con:Roots_of_S} holds).
\end{cor}

\begin{proof}
The first implication is obvious.
For the reciprocal implication,
we simply recall that $S(x)$ has only real roots when $m \le n+3$
and all its real roots are contained in
$\{ x \in \Rset : R(x) < 0 \} = \bigcup_{r=1}^n (\gamma_{2r},\gamma_{2r-1})$;
see Proposition~\ref{pro:Polynomial}.
\end{proof}

\begin{defi}\label{defi:Decompositions}
Given any signature $\tau \in \mathcal{T}(m,n)$,
let $\{1,\ldots,2n-1\} = J_\tau \cup K_\tau$ and
$\{1,\ldots,m-n\} = V_\tau \cup W_\tau$ be the decompositions determined
as follows.
If $\delta_{m-n} < \cdots < \delta_1$ is any ordered sequence
verifying~(\ref{eq:tau}), then the elements of the multisets
\begin{eqnarray*}
\{ \alpha_1,\ldots,\alpha_m \} & = &
\{ \gamma_j : j \in J_\tau \} \cup \{ \delta_v, \delta_v : v \in V_\tau \}, \\
\{ \beta_1,\ldots,\beta_{m-1} \} & = &
\{ \gamma_k : k \in K_\tau \} \cup \{ \delta_w, \delta_w : w \in W_\tau \},
\end{eqnarray*}
are ordered as in Lemma~\ref{lem:Ordering}.
\end{defi}

Multisets are a generalization of sets in which members are allowed
to appear more than once; see~\cite{Knuth1998}.
In our case, the numbers $\delta_1,\ldots,\delta_{m-n}$ appear twice.

These decompositions are well-defined.
That is, they only depend on the signature $\tau$,
since any ordered sequence $\delta_{m-n} < \cdots < \delta_1$
verifying~(\ref{eq:tau}) gives rise to the same decomposition.
The decomposition $\{1,\ldots,2n-1\} = J_n \cup K_n$ given in
Corollary~\ref{cor:JK} correspond to the trivial signature
$\tau = (0,\ldots,0) \in \mathcal{T}(n,n)$.

Next, we generalize Corollary~\ref{cor:JK}
and Proposition~\ref{pro:C_n_n} to the case $m \ge n$.

\begin{cor}\label{cor:JKVW}
Let $\delta_{m-n} < \cdots < \delta_1$ be an ordered sequence
verifying~(\ref{eq:tau}) for some signature $\tau \in \mathcal{T}(m,n)$.
If $P(x)$ is a monic polynomial of degree $m$ such that
$P(0) \neq 0$ and
\[
\textstyle
\prod_{u=1}^{m-n} (x-\delta_u)^2 \cdot R(x) =
P(x) (P(x) - P(0)),
\]
then
\(
P(x) =
\prod_{j \in J_\tau} (x-\gamma_j) \prod_{v \in V_\tau} (x-\delta_v)^2 =
P(0) + x \prod_{k \in K_\tau} (x-\gamma_k) \prod_{w \in W_\tau} (x-\delta_w)^2
\).
\end{cor}

\begin{proof}
There exist two decompositions $\{1,\ldots,2n-1\} = J' \cup K'$
and $\{1,\ldots,m-n\} = V' \cup W'$ such that
$P(x) =
 \prod_{j \in J'} (x-\gamma_j) \prod_{v \in V'} (x-\delta_v)^2 =
 P(0) + x \prod_{k \in K'} (x-\gamma_k) \prod_{w \in W'} (x-\delta_w)^2$.
The polynomial $P(x)$ verifies the hypotheses stated in Lemma~\ref{lem:Ordering},
so the roots
\begin{eqnarray*}
\{ \alpha_1, \ldots, \alpha_m \} & = &
\{ \gamma_j : j \in J' \} \cup \{ \delta_v, \delta_v : v \in V' \} \\
\{ \beta_1, \ldots, \beta_{m-1} \} & = &
\{ \gamma_k : k \in K' \} \cup \{ \delta_w, \delta_w : w \in W' \}
\end{eqnarray*}
obey the ordering described in that lemma.
Hence, $J' = J_\tau$, $K' = K_\tau$, $V' = V_\tau$, and $W' = W_\tau$.
\end{proof}

\begin{pro}\label{pro:C_m_n_tau}
Condition $\mathcal{C}(m,n;\tau)$ holds if and only if there exists a
sequence $\delta_{m-n} < \cdots < \delta_1$ verifying~(\ref{eq:tau})
such that the following three equivalent properties hold:
\begin{enumerate}[(i)]
\item
If
$P(x) = \prod_{j \in J_\tau} (x-\gamma_j) \prod_{v \in V_\tau} (x-\delta_v)^2$, then
$P(x) - P(0) = x \prod_{k \in K_\tau } (x-\gamma_k) \prod_{w \in W_\tau} (x-\delta_w)^2$.
\item
$\rme_l \big( \{\gamma_j\}_{j \in J_\tau} \cup \{ \delta_v,\delta_v \}_{v \in V_\tau} \big) =
 \rme_l \big( \{\gamma_k \}_{k \in K_\tau} \cup \{ \delta_w,\delta_w \}_{w \in W_\tau} \big)$,
for all $l=1,\ldots,m-1$.
\item
$\sum_{j \in J_\tau} \gamma_j^l + 2 \sum_{v \in V_\tau} \delta_v^l =
 \sum_{k \in K_\tau} \gamma_k^l + 2 \sum_{w \in W_\tau} \delta_w^l$,
for all $l = 1,\ldots,m-1$.
\end{enumerate}
\end{pro}

\begin{proof}
We simply repeat the steps of the proof of Proposition~\ref{pro:C_n_n},
but using Corollary~\ref{cor:JKVW} instead of Corollary~\ref{cor:JK}.
\end{proof}

\begin{example}\label{ex:Planar2}
The quantities $\gamma_3 = 1$, $\gamma_2 = 4$, and $\gamma_1 = 9$
verify condition $\mathcal{C}(3,2;\tau)$ with $\tau = (1,0)$,
because $1 + 4 + 9 = 2 \cdot 7$, $1^2 + 4^2 + 9^2 = 2 \cdot 7^2$,
and $7 \in (4,9)$.
Hence, the billiard trajectories
\begin{itemize}
\item
Inside the ellipse $Q : x^2 + 4 y^2 = 1$ with caustic parameter
$\lambda = 1/9$; or
\item
Inside the ellipse $Q : x^2 + 9 y^2 = 1$ with caustic parameter
$\lambda = 1/4$;
\end{itemize}
are periodic with elliptic periodic $m = 3$.
Their caustic types are $\varsigma = 0$ and $\varsigma = 1$, respectively.
\end{example}

All conditions $\mathcal{C}(m,n;\tau)$, $\tau \in \mathcal{T}(m,n)$,
give rise to nonsingular periodic billiard trajectories with
elliptic period $m$,
so we wondered which is the dynamical meaning of the signature $\tau$.
We believe that there exists a one-to-one correspondence between
the elliptic winding numbers $\widetilde{m}_0,\ldots,\widetilde{m}_{n-1}$
---see Definition~\ref{defi:EllipticPeriod}---
and the signature $\tau = (\tau_1,\ldots,\tau_n)$.

\begin{con}\label{con:tau}
Set $\widetilde{m}_n = 0$. Then
$\widetilde{m}_j =  \widetilde{m}_{j+1} + \tau_{j+1} + 1$ for all $j = 0,\ldots,n-1$.
\end{con}

This conjecture follows from the interpretation of $\mathcal{C}(m,n;\tau)$
as a singular limit of $\mathcal{C}(m,m)$ when $m-n$ couples of simple roots
collide, so they become double roots.
Unfortunately,
we have not been able to transform this argument into a rigorous proof,
although all our analytical and numerical computations agree with
the conjecture.

To end this section,
we stress that if conjectures~\ref{con:Roots_of_S} and~\ref{con:tau} hold,
then the elliptic winding numbers $\widetilde{m}_0,\ldots,\widetilde{m}_{n-1}$
of any nonsingular periodic billiard trajectory verify the above-mentioned
relations for some signature
$\tau = (\tau_1,\ldots,\tau_n)$ with non-negative entries,
so the sequence $\widetilde{m}_0,\ldots,\widetilde{m}_{n-1}$ strictly decreases,
and Conjecture~\ref{con:WindingNumbers} holds.

\section{The planar case}
\label{sec:Planar}

We adapt the previous setting of billiards inside ellipsoids of $\Rset^n$
to the planar case $n = 2$.
To follow traditional conventions in the literature,
we write the ellipse as
\begin{equation}\label{eq:Ellipse}
Q =
\left\{
(x,y) \in \Rset^2 : \frac{x^2}{a} + \frac{y^2}{b} = 1
\right\},\qquad
a > b > 0.
\end{equation}
Any nonsingular billiard trajectory inside $Q$ is tangent
to one confocal caustic
\[
Q_\lambda =
\left\{
(x,y) \in \Rset^2 :
\frac{x^2}{a - \lambda} + \frac{y^2}{b - \lambda} = 1
\right\},
\]
where $\lambda \in \Lambda = E \cup H$, with
$E = \left( 0, b \right)$ and $H = \left( b, a \right)$.

The names of the connected components of $\Lambda$ come from the fact that
$Q_\lambda$ is a confocal ellipse for $\lambda \in E$ and
a confocal hyperbola for $\lambda \in H$.
The singular cases $\lambda = b$ and $\lambda = a$ correspond to the
$x$-axis and $y$-axis, respectively.
We say that the \textit{caustic type} of a billiard trajectory is E or H,
when its caustic is an ellipse or a hyperbola (compare with
Definition~\ref{defi:CausticType}).
We also distinguish between E-caustics and H-caustics.

We recall some concepts related to periodic trajectories of
billiards inside ellipses.
These results can be found, for instance,
in~\cite{ChangFriedberg1988,CasasRamirez2011}.
To begin with,
we introduce the function $\rho : \Lambda \rightarrow \Rset$ given by
the quotient of elliptic integrals
\begin{equation}\label{eq:RotationNumber}
\rho(\lambda) = \rho(\lambda;b,a) :=
\frac{\int_0^{\min(b,\lambda)}\frac{\rmd t}{\sqrt{(\lambda-t)(b-t)(a-t)}}}
     {2 \int_{\max(b,\lambda)}^a \frac{\rmd t}{\sqrt{(\lambda-t)(b-t)(a-t)}}}.
\end{equation}
It is called the \textit{rotation number} and characterizes
the caustic parameters that give rise to periodic trajectories.
To be precise, the billiard trajectories with caustic
$Q_\lambda$ are periodic if and only if
\[
\rho(\lambda) = m_1 / 2 m_0 \in \Qset
\]
for some integers $2 \leq m_1 < m_0$,
which are the \textit{winding numbers}.
On the one hand, $m_0$ is the period,
On the other hand,
$m_1$ is twice the number of turns around the ellipse $Q_\lambda$ for E-caustics,
and the number of crossings of the $y$-axis for H-caustics.
Thus, $m_1$ is always even.
Besides, all periodic trajectories with H-caustics have even period.
(Compare with Theorem~\ref{thm:WindingNumbers}.)

\begin{pro}
The winding numbers $2 \le m_1 < m_0$,
rotation number $\rho = m_1 / 2 m_0$,
signature $\tau = (\tau_1,\tau_2) \in \mathcal{T}(m,2)$,
caustic type (E or H), and caustic parameter $\lambda$
of all nonsingular periodic billiard trajectories inside the
ellipse~(\ref{eq:Ellipse}) with elliptic period $m \in \{2,3\}$
are listed in Table~\ref{tab:Planar}.
The ellipses where such trajectories take place are also listed.
\end{pro}

\begin{table}
\centering
\begin{tabular}{cccccccc}
\hline
$m$ & $m_0$ & $m_1$ & $\rho$ & $\tau$  & Type & Ellipses  & Caustic parameter \\
\hline
$2$ & $4$  & $2$   & $1/4$  & $(0,0)$ & E & any      & $\frac{ab}{a+b}$ \\
$2$ & $4$  & $2$   & $1/4$  & $(0,0)$ & H & $2b<a$   & $\frac{ab}{a-b}$ \\
$3$ & $6$  & $2$   & $1/6$  & $(1,0)$ & E & any      & $\frac{ab}{a+b+2\sqrt{ab}}$ \\
$3$ & $6$  & $2$   & $1/6$  & $(1,0)$ & H & $4b<a$   & $\frac{ab}{a+b-2\sqrt{ab}}$ \\
$3$ & $3$  & $2$   & $1/3$  & $(0,1)$ & E & any      & $\frac{3ab}{a+b+2\sqrt{a^2-ab+b^2}}$ \\
$3$ & $6$  & $4$   & $1/3$  & $(0,1)$ & H & $4b<3a$  & $\frac{ab}{2\sqrt{a^2-ab}+b-a}$ \\
\hline
\end{tabular}
\caption{\label{tab:Planar}Algebraic formulas for the caustic parameter
corresponding to nonsingular periodic billiard trajectories with
elliptic period $m \in \{2,3\}$ in the planar case.}
\end{table}

\begin{proof}
We split the proof in four steps.

\textit{Step 1: To find the solutions of $\mathcal{C}(m,2)$
in terms of the inverse quantities $\gamma_i$.}
First, we saw in Proposition~\ref{pro:C_n_n} that $\mathcal{C}(2,2)$ holds
if and only if $\gamma_1 = \gamma_2 + \gamma_3$.

Next, we focus on the case $m=3$.
We note that $\mathcal{C}(3,2)$ holds if and only if
$\mathcal{C}(3,2;\tau)$ holds for some $\tau = (\tau_1,\tau_2) \in \Zset^2$
such that $\tau_1 + \tau_2 = 1$ and $\tau_1,\tau_2 \ge 0$;
see Corollary~\ref{cor:C_m_n_tau}.

Let us begin with the signature $\tau = (1,0)$.
After a straightforward check,
we get that the decompositions presented in Definition~\ref{defi:Decompositions}
are $J_\tau=\{1,2,3\}$, $K_\tau = V_\tau = \emptyset$,
and $W_\tau = \{ 1 \}$.
Thus, $\mathcal{C}(3,2;\tau)$ holds if and only if there exists some
$\delta_1 \in (\gamma_2,\gamma_1)$ such that
\[
P(x) = (x-\gamma_1)(x-\gamma_2)(x-\gamma_3) = P(0) + x(x-\delta_1)^2,
\]
or, equivalently, if and only if the discriminant of the polynomial
\[
Q(x) = \frac{P(x) - P(0)}{x} =
x^2 - \rme_1(\gamma_1,\gamma_2,\gamma_3) x + \rme_2(\gamma_1,\gamma_2,\gamma_3)
\]
is equal to zero.
The discriminant of $Q(x)$ is
\[
\Delta =
\gamma_1^2 + \gamma_2^2 + \gamma_3^2 -
2 \gamma_1 \gamma_2 - 2 \gamma_1 \gamma_3 - 2\gamma_2 \gamma_3.
\]
We already saw in Section~\ref{sec:Practical} that
the only feasible solution of $\Delta = 0$ is
$\sqrt{\gamma_1} = \sqrt{\gamma_2} + \sqrt{\gamma_3}$.

When $\tau = (0,1)$, the decompositions are
$J_\tau=\{ 1 \}$, $K_\tau = \{2,3\}, V_\tau = \{1\}$,
and $W_\tau = \emptyset$.
Thus, $\mathcal{C}(3,2;\tau)$ holds if and only if there exists some
$\delta_1 \in (0,\gamma_3)$ such that
\[
\gamma_1 + 2\delta_1 = \gamma_2 + \gamma_3,\qquad
\gamma^2_1 + 2\delta_1^2 = \gamma^2_2 + \gamma^2_3,
\]
or, equivalently, if and only if
\begin{eqnarray*}
3 \gamma^2_1 - 2(\gamma_2 + \gamma_3) \gamma_1 - (\gamma_2  - \gamma_3)^2
& = &
(\gamma_2 + \gamma_3 -\gamma_1)^2 - 2(\gamma^2_2 + \gamma^2_3 - \gamma^2_1) \\
& = &
(2 \delta_1)^2 - 4\delta_1^2 = 0.
\end{eqnarray*}
And we already saw in Section~\ref{sec:Practical} that
the only feasible solution of the above equation is
\[
3\gamma_1 =
\gamma_2 + \gamma_3 + 2\sqrt{\gamma^2_2 + \gamma^2_3 - \gamma_2 \gamma_3}.
\]

\textit{Step 2:
        To express the above solutions in terms of $a$, $b$, and $\lambda$}.
If the caustic type is E,
then $\lambda \in (0,b)$, $\gamma_1 = 1/\lambda$, $\gamma_2 = 1/b$,
and $\gamma_3 = 1/a$.
Thus,
\begin{eqnarray*}
\gamma_1 = \gamma_2 + \gamma_3
& \Leftrightarrow &
\lambda = \textstyle \frac{ab}{a+b}, \\
\sqrt{\gamma_1} = \sqrt{\gamma_2} + \sqrt{\gamma_3}
& \Leftrightarrow &
\lambda = \textstyle \frac{ab}{a+b+2\sqrt{ab}} \\
3\gamma_1 = \gamma_2 + \gamma_3 +
            2\sqrt{\gamma^2_2 + \gamma^2_3 - \gamma_2 \gamma_3}
& \Leftrightarrow &
\textstyle \lambda = \frac{3ab}{a+b+2\sqrt{a^2-ab+b^2}}.
\end{eqnarray*}

If the caustic type is H,
then $\lambda \in (b,a)$, $\gamma_1 = 1/b$, $\gamma_2 = 1/\lambda$,
and $\gamma_3 = 1/a$.
Thus,
\begin{eqnarray*}
\gamma_1 = \gamma_2 + \gamma_3
& \Leftrightarrow &
\textstyle \lambda = \frac{ab}{a-b}, \\
\sqrt{\gamma_1} = \sqrt{\gamma_2} + \sqrt{\gamma_3}
& \Leftrightarrow &
\lambda = \textstyle \frac{ab}{a+b - 2\sqrt{ab}} \\
3\gamma_1 = \gamma_2 + \gamma_3 +
            2\sqrt{\gamma^2_2 + \gamma^2_3 - \gamma_2 \gamma_3}
& \Leftrightarrow &
\textstyle \lambda = \frac{ab}{2\sqrt{a^2-ab} + b -a}.
\end{eqnarray*}

\textit{Step 3: To determine the ellipses where such
periodic billiard trajectories take place.}
We ask whether the caustic parameters found above
belong to the interval $(0,b)$ for E-caustics,
and to the interval $(b,a)$ for H-caustics.
The caustic type E does not give any restriction, because
\[
\textstyle
0 < b < a \mbox{ and }
\lambda \in
\left\{
\frac{ab}{a+b},
\frac{ab}{a+b+2\sqrt{ab}},
\frac{3ab}{a+b+2\sqrt{a^2-ab+b^2}}
\right\} \Rightarrow \lambda \in (0,b).
\]
On the contrary, the caustic type H gives rise to some restrictions.
Namely,
\begin{eqnarray*}
\textstyle b < \frac{ab}{a-b} < a & \Leftrightarrow & 2b < a,\\
\textstyle b < \frac{ab}{a+b - 2\sqrt{ab}} < a & \Leftrightarrow & 4b < a, \\
\textstyle b < \frac{ab}{2\sqrt{a^2-ab} + b -a} < a & \Leftrightarrow & 4b < 3a.
\end{eqnarray*}

\begin{figure}
\centering
\subfigure[$\lambda= \frac{ab}{a+b}$]{
  \label{fig:Planar:a}
  \includegraphics[width=1.68in]{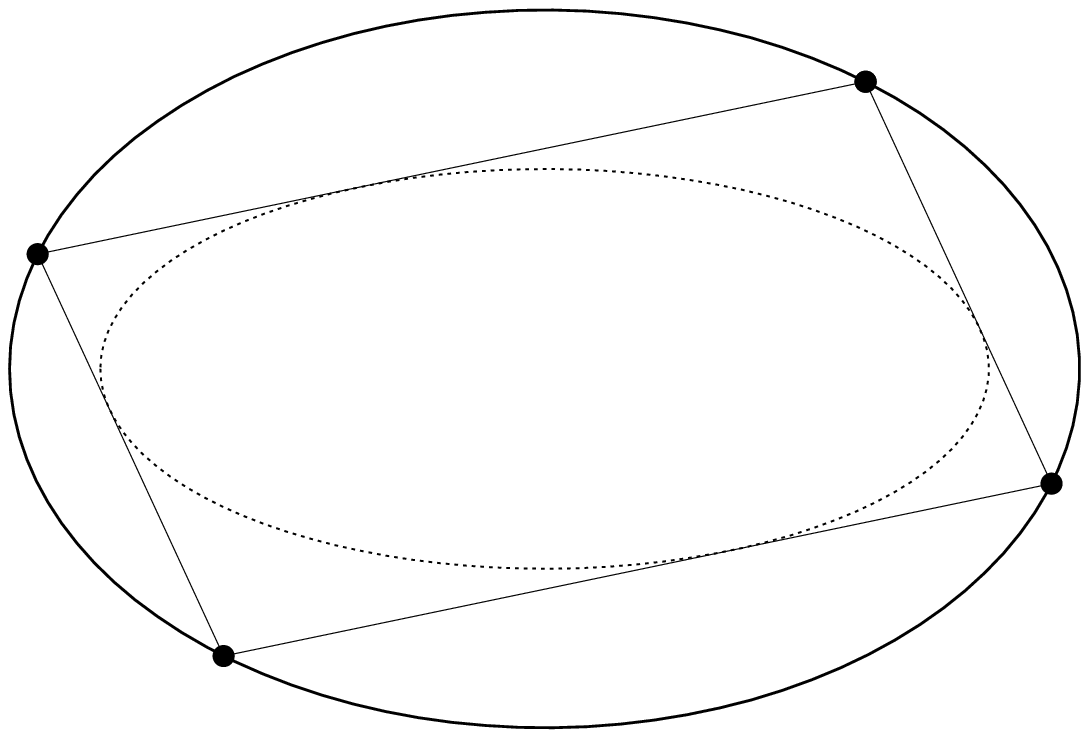}
}
\hfill
\subfigure[$\lambda = \frac{ab}{a+b+2\sqrt{ab}}$]{
  \label{fig:Planar:b}
  \includegraphics[width=1.68in]{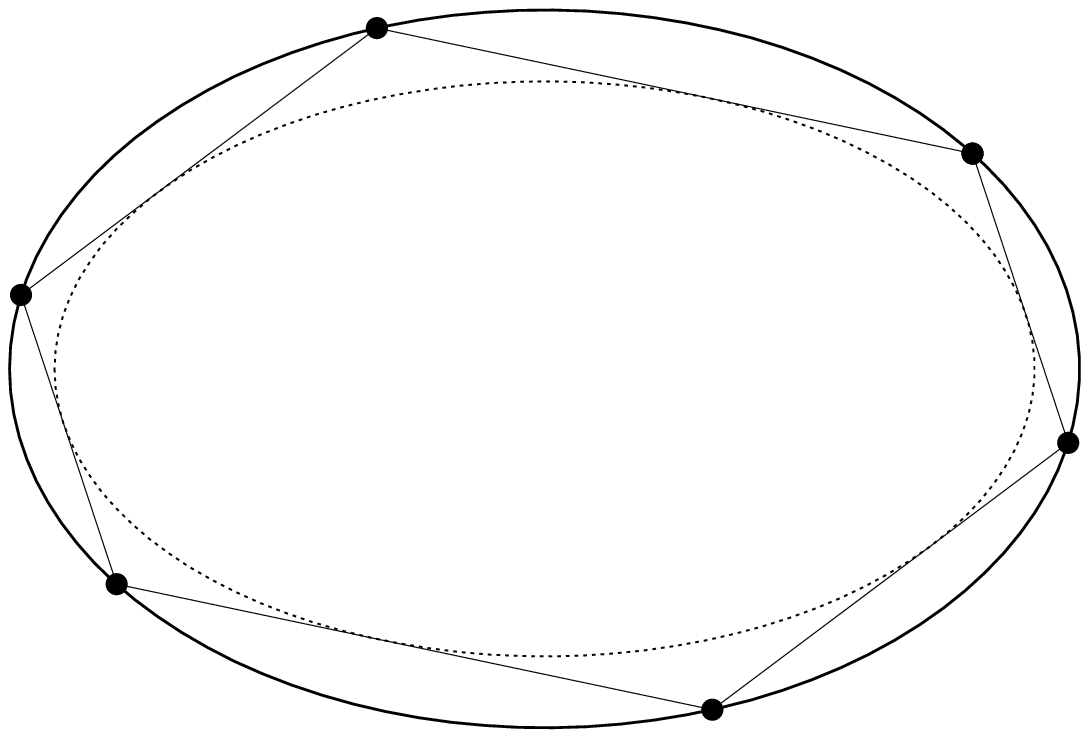}
}
\hfill
\subfigure[$\lambda = \frac{3ab}{a+b+2\sqrt{a^2-ab+b^2}}$]{
  \label{fig:Planar:c}
  \includegraphics[width=1.68in]{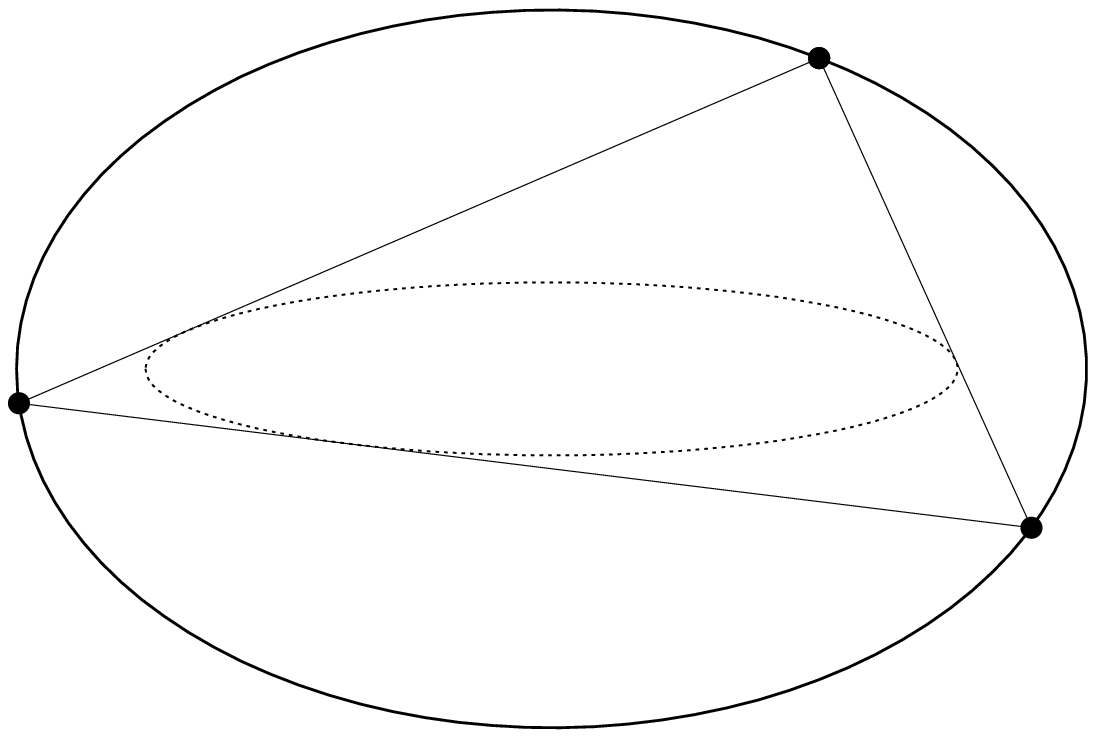}
}

\subfigure[$\lambda = \frac{ab}{a-b}$]{
  \label{fig:Planar:d}
  \includegraphics[width=1.68in]{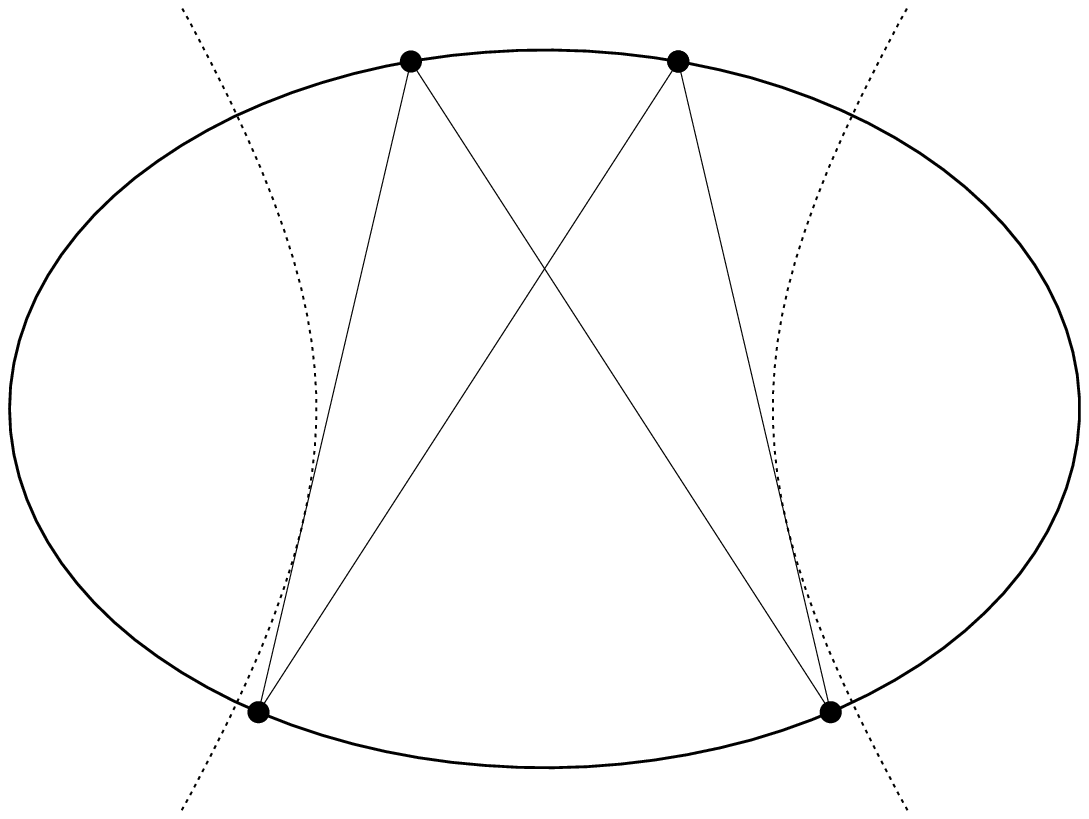}
}
\hfill
\subfigure[$\lambda = \frac{ab}{a+b-2\sqrt{ab}}$]{
  \label{fig:Planar:e}
  \includegraphics[width=1.68in]{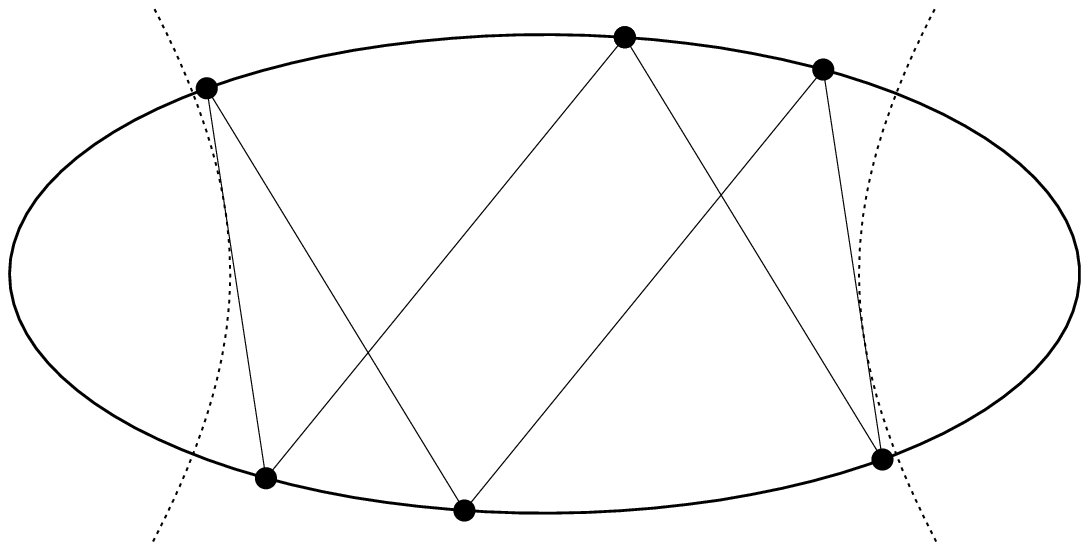}
}
\hfill
\subfigure[$\lambda = \frac{ab}{2\sqrt{a^2-ab}+b-a}$]{
  \label{fig:Planar:f}
  \includegraphics[width=1.68in]{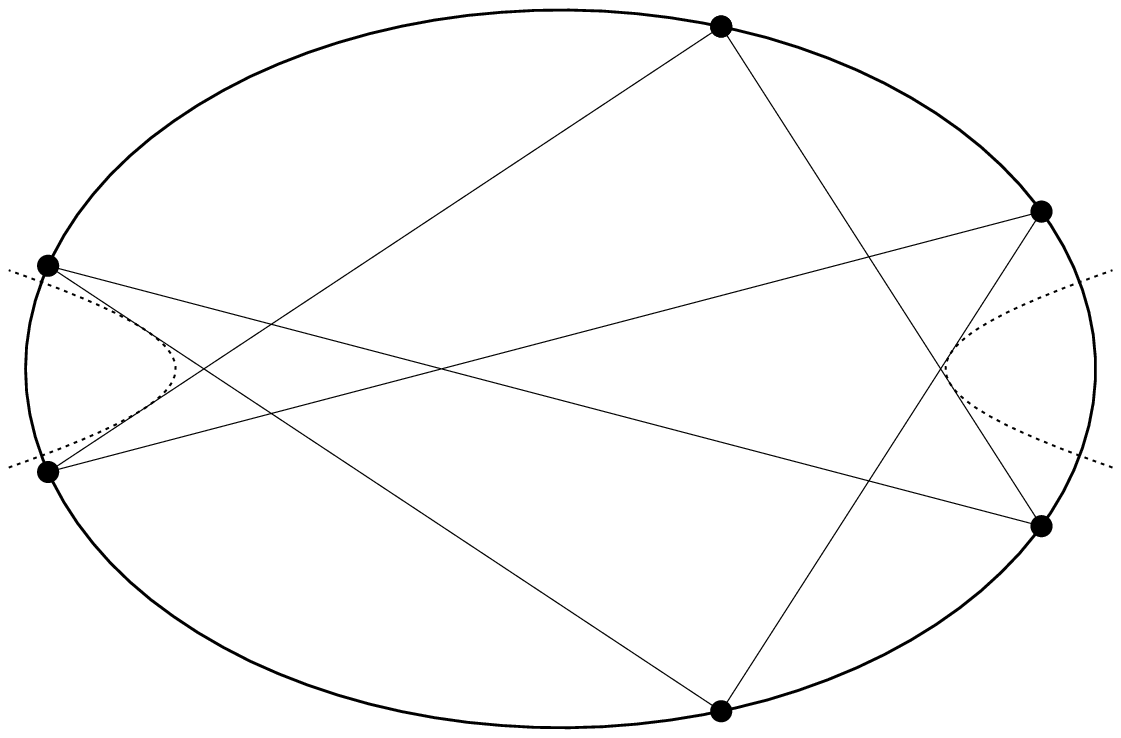}
}
\caption{Periodic trajectories corresponding to the caustic
         parameters listed in Table~\ref{tab:Planar}.
         The ellipse for $\lambda = \frac{ab}{a+b-2\sqrt{ab}}$ is flatter,
         because it must satisfy condition $4b < a$.}
\label{fig:Planar}
\end{figure}

\textit{Step 4: To find the winding numbers and the rotation number.}
The winding numbers $2 \le m_1 < m_0$ and the rotation number
$\rho(\lambda) = m_1/2m_0$ are obtained from geometric arguments.
To be precise,
we draw in Figure~\ref{fig:Planar} a billiard trajectory tangent
to $Q_{\lambda}$ for each of the caustic parameters listed in
Table~\ref{tab:Planar}.
Then we recall that $m_0$ is the period and $m_1$ is twice the number of
turns around the ellipse $Q_\lambda$ for E-caustics,
and the number of crossings of the $y$-axis for H-caustics.
\end{proof}

Let $\rho_\ast \in \{1/3,1/4,1/6\}$.
We have seen above that inside any ellipse~(\ref{eq:Ellipse}) there exists
a unique E-caustic whose tangent billiard trajectories have
rotation number $\rho_\ast$.
Besides,
if $\lambda_\rmE(a,b;\rho_\ast)$ and $\lambda_\rmH(a,b;\rho_\ast)$
denote the caustic parameters associated to the E-caustic
and H-caustic with rotation number $\rho_\ast$, we see that
\begin{equation}\label{eq:Symmetries}
\lambda_\rmE(b,a;\rho_\ast) = \lambda_\rmE(a,b;\rho_\ast),\qquad
b=\lambda_\rmE(a,\lambda_\rmH(a,b;\rho_\ast);\rho_\ast).
\end{equation}

These properties can be generalized.
On the one hand,
the rotation number~(\ref{eq:RotationNumber}) is an increasing function
in the interval $(0,b)$ such that $\rho(0) = 0$ and $\rho(b) = 1/2$;
see~\cite{CasasRamirez2011}.
This means that given any $\rho_\ast \in(0,1/2)$,
there exists a unique $\lambda_\ast \in (0,b)$ such that
$\rho(\lambda_\ast) = \rho_\ast$, and so,
there exists a unique E-caustic whose tangent billiard trajectories
have rotation number $\rho_\ast$.
On the other hand,
relations~(\ref{eq:Symmetries}) can be obtained by using that the rotation
number~(\ref{eq:RotationNumber}) is symmetric in the three parameters
$a$, $b$, and $\lambda$.
Consequently,
one can find the formula for $\lambda_\rmH$ (respectively, $\lambda_\rmE$)
from the formula for $\lambda_\rmE$ (respectively, $\lambda_\rmH$)
by using the second relation.

It is interesting to realize that the results in Table~\ref{tab:Planar}
agree with Conjecture~\ref{con:tau}.

In the planar case $n=2$,
the caustic type~(\ref{eq:CausticType}) is $\varsigma = 0$ or,
equivalently, E.
Hence, the planar version of Theorem~\ref{thm:Main} is shown in
the first row of Table~\ref{tab:Planar},
because $\lambda = \lambda_\rmE(a,b;1/4) = ab/(a+b)$
is the root of $t^2 - (t-a)(t-b)$.
This naive observation was the germ of this paper.

\section{The spatial case}
\label{sec:Spatial}

In order to study the spatial case $n = 3$,
we consider the triaxial ellipsoid
\begin{equation}\label{eq:Ellipsoid_3D}
Q =
\left\{
(x,y,z) \in \Rset^3 :
\frac{x^2}{a} + \frac{y^2}{b} + \frac{z^2}{c} = 1
\right\},
\qquad a > b > c > 0.
\end{equation}

Any nonsingular billiard trajectory inside $Q$ is
tangent to two distinct nonsingular caustics
$Q_{\lambda_1}$ and $Q_{\lambda_2}$, with $\lambda_1 < \lambda_2$,
of the confocal family
\begin{equation}\label{eq:ConfocalFamily}
Q_\lambda =
\left\{
(x,y,z) \in \Rset^3 :
\frac{x^2}{a-\lambda} + \frac{y^2}{b-\lambda} + \frac{z^2}{c-\lambda}
= 1
\right\}.
\end{equation}
The caustic $Q_\lambda$ is an ellipsoid for $\lambda \in (0,c)$,
a hyperboloid of one sheet when $\lambda \in (c,b)$,
and a hyperboloid of two sheets if $\lambda \in (b,a)$.
Not all combinations of nonsingular caustics can take place,
but only the four caustic types EH1, H1H1, EH2, and H1H2.

\begin{pro}
The caustic type and caustic parameters of all nonsingular periodic
billiard trajectories inside the triaxial ellipsoid~(\ref{eq:Ellipsoid_3D})
with elliptic period $m = 3$ are listed in Table~\ref{tab:Spatial}.
The ellipsoids where such trajectories take place are also listed.
\end{pro}

\begin{table}
\centering
\begin{tabular}{lll}
\hline
Type  & Ellipsoids  & Caustic parameters \\
\hline
EH1 &
$c < \displaystyle{\frac{ab}{a+b+\sqrt{ab}}}$ &
$\left\{\begin{array}{c}
c^3 = (c-\lambda_1)(b-c)(a-c) \\
1/\lambda_2 + 1/c = 1/a + 1/b + 1/\lambda_1
\end{array}\right.$ \\
H1H1 &
$c < \displaystyle{\frac{ab}{a+b+2\sqrt{ab}}}$ &
Roots of $t^3-(t-a)(t-b)(t-c)$ \\
EH2 &
$\left\{ \begin{array}{l}
 c < \displaystyle{\frac{a-2b}{2a-3b}}a \\
 2b < a
 \end{array}\right.$ &
Roots of $(a-b)(a-c)t^2 + ( bc - a(b+c) ) a t + a^2 b c$ \\
H1H2 &
$\left\{ \begin{array}{l}
c < \displaystyle{\frac{a-2b}{(a-b)^2}}ab \\
b > \displaystyle{\frac{ac}{a+c-\sqrt{ac}}}
  \end{array} \right.$ &
$\left\{\begin{array}{c}
b^3 = (b-c)(\lambda_2-b)(a-b) \\
1/\lambda_1 + 1/b = 1/a + 1/c + 1/\lambda_2
\end{array}\right.$ \\
\hline
\end{tabular}
\caption{\label{tab:Spatial}Algebraic formulas for the caustic parameters
corresponding to the nonsingular periodic trajectories with
elliptic period $m = 3$ in the spatial case.}
\end{table}

\begin{proof}
If $a$, $b$, and $c$ are the ellipsoidal parameters,
and $\lambda_1$ and $\lambda_2$ are the caustic parameters, we set
$\{ c_1,c_2,c_3,c_4,c_5 \} = \{ a,b,c,\lambda_1,\lambda_2\}$,
where $0 < c_1 < c_2 < c_3 < c_4 < c_5$.
We also set $\gamma_i = 1/c_i$.
Let $\{1,2,3,4,5\} = J \cup K$, with $J = \{1,4,5\}$ and $K = \{2,3\}$,
be the decomposition defined in Corollary~\ref{cor:JK} when $n=3$.
From Proposition~\ref{pro:C_n_n} we know that
\begin{eqnarray*}
\mathcal{C}(3,3) & \Leftrightarrow &
\gamma_2 + \gamma_3 = \gamma_1 + \gamma_4 + \gamma_5 \mbox{ and }
\gamma^2_2 + \gamma^2_3 = \gamma^2_1 + \gamma^2_4 + \gamma^2_5 \\
& \Leftrightarrow &
(\gamma_1 - \gamma_k)(\gamma_4 - \gamma_k)(\gamma_5 - \gamma_k) =
\gamma_1 \gamma_4 \gamma_5, \mbox{ for $k=2,3$} \\
& \Leftrightarrow &
c^3_k = (c_k - c_1)(c_4 - c_k)(c_5 - c_k), \mbox{ for $k=2,3$.}
\end{eqnarray*}

In the rest of the proof, we study each caustic type separately.

\textit{Caustic type EH1.}
In this case $0 < \lambda_1 < c < \lambda_2 < b < a$, so
\[
c_1 = \lambda_1,\qquad
c_2 = c,\qquad
c_3 = \lambda_2,\qquad
c_4 = b,\qquad
c_5 = a.
\]
Thus the formula for $\lambda_1$ follows from relation
$c^3_2 = (c_2 - c_1)(c_4 - c_2)(c_5 - c_2)$,
whereas the formula for $\lambda_2$  follows from relation
$\gamma_2 + \gamma_3 = \gamma_1 + \gamma_4 + \gamma_5$.
Next, we look for ellipsoidal parameters such that
the caustic parameters computed using these two formulas are
placed in the right intervals:
$\lambda_1 \in (0,c)$ and $\lambda_2 \in (c,b)$.

To begin with, we note that $\lambda_1 < c$,
since $(c-\lambda_1)(b-c)(a-c) = c^3 > 0$.
Besides,
\[
\lambda_1 = \frac{ab-(a+b)c}{(b-c)(a-c)}c > 0 \Leftrightarrow
c < \frac{ab}{a+b}.
\]
On the other hand, if $\lambda_1 \in (0,c)$, then
\[
1/\lambda_2 = 1/a + 1/b + (1/\lambda_1 - 1/c) > 1/a + 1/b > 1/b,
\]
so $\lambda_2 < b$. Finally,
\begin{eqnarray*}
\lambda_2 > c & \Leftrightarrow &
\frac{1}{c} + \frac{c}{ab-(a+b)c} =
\frac{1}{\lambda_1} =
\frac{1}{\lambda_2} + \frac{1}{c} - \frac{1}{a} - \frac{1}{b}
< \frac{2}{c} - \frac{1}{a} - \frac{1}{b} \\
& \Leftrightarrow &
c < \frac{ab}{a+b+\sqrt{ab}}.
\end{eqnarray*}
Therefore, $\lambda_1 \in (0,c)$ and $\lambda_2 \in (c,b)$
if and only if $c < ab/(a+b+\sqrt{ab})$.

\textit{Caustic type H1H1.}
If $n=3$,
then the caustic type~(\ref{eq:CausticType}) is $\varsigma = (1,1)$ or,
equivalently, H1H1.
Hence, the study for the caustic type H1H1 was already carried out
in Theorem~\ref{thm:Main}.
Suffice it to note that the polynomial
\[
t^3-(t-a)(t-b)(t-c) = (a+b+c)t^2 - (ab+ac+bc) t + abc
\]
has two real simple roots if and only if its discriminant
\[
\Delta = (ab+ac+bc)^2 - 4abc(a+b+c) =
(a-b)^2 c^2 - 2ab(a + b)c + a^2b^2
\]
is positive.
This discriminant is a second-degree polynomial in $c$
whose roots are
\[
c_\pm =
\frac{ab(a+b) \pm 2ab \sqrt{ab}}{(a-b)^2} =
\frac{ab}{a+b \mp 2 \sqrt{ab}}.
\]
We note that $0 < c_- < b < c_+$.
Thus, using that $0 < c < b < a$,
we get $\Delta > 0 \Leftrightarrow c < c_-$.

\textit{Caustic type EH2.}
In this case $0 < \lambda_1 < c < b < \lambda_2 < a$, so
\[
c_1 = \lambda_1,\qquad
c_2 = c,\qquad
c_3 = b,\qquad
c_4 = \lambda_2,\qquad
c_5 = a.
\]
Using relations
$\gamma^l_2 + \gamma^l_3 = \gamma^l_1 + \gamma^l_4 + \gamma^l_5$,
with $l=1,2$, we know that
\[
s_l := \frac{1}{\lambda^l_1} + \frac{1}{\lambda^l_2} =
\frac{1}{c^l} + \frac{1}{b^l} - \frac{1}{a^l},\qquad l =1,2.
\]
Hence, $1/\lambda_1$ and $1/\lambda_2$ are the roots of the polynomial
\[
(x-1/\lambda_1)(x-1/\lambda_2) =
x^2 - s_1 x + \frac{s_1^2-s_2}{2} =
x^2 + \frac{bc-a(b+c)}{abc}x + \frac{(a-b)(a-c)}{a^2 b c}.
\]
Thus, using the change of variables $t=1/x$,
we get that $\lambda_1$ and $\lambda_2$ are the roots of
\[
Q(t) = (a-b)(a-c)t^2 + ( bc - a(b+c) ) a t + a^2 b c.
\]
We look for ellipsoidal parameters such that $Q(t)$ has a root in $(0,c)$
and a root in $(b,a)$.
The root in $(0,c)$ always exists,
since $Q(0) = a^2 b c > 0$ and $Q(c) = - c^3(a-b) < 0$.
Besides, $Q(b) = -b^3(a-c) < 0$
and $\lim_{t \to + \infty} Q(t) = +\infty$,
so $Q(t)$ has a root in $(b,a)$ if and only if
\[
Q(a) = a^2 \left(a^2 - 2a(b+c) + 3bc \right) > 0,
\]
or, equivalently, if and only if $c < (a-2b)a/(2a-3b)$ and $2b < a$.
We have used that $0 < c < b < a$ in the last equivalence.

\textit{Caustic type H1H2.}
In this case $0 < c < \lambda_1 < b < \lambda_2 < a$, so
\[
c_1 = c,\qquad
c_2 = \lambda_1,\qquad
c_3 = b,\qquad
c_4 = \lambda_2,\qquad
c_5 = a.
\]
Thus the formula for $\lambda_2$ follows from relation
$c^3_2 = (c_2 - c_1)(c_4 - c_2)(c_5 - c_2)$,
whereas the formula for $\lambda_1$  follows from relation
$\gamma_2 + \gamma_3 = \gamma_1 + \gamma_4 + \gamma_5$.
Next, we look for conditions on the ellipsoidal parameters such that
the caustic parameters computed from the previous formulas are
placed in the right intervals:
$\lambda_1 \in (c,b)$ and $\lambda_2 \in (b,a)$.

To begin with, we note that $\lambda_2 > b$,
because $(b-c)(\lambda_2-b)(a-b) = b^3 > 0$.
Besides,
\[
\frac{(a+c)b-ac}{(a-b)(b-c)}b = \lambda_2 < a \Leftrightarrow
c < \frac{a-2b}{(a-b)^2}ab.
\]
On the other hand, using that $0 < c < b < a$, we get that
\begin{eqnarray*}
c < \lambda_1 < b & \Leftrightarrow &
\frac{2}{b} - \frac{1}{a} - \frac{1}{c} <
\frac{1}{\lambda_2} = \frac{1}{\lambda_1} + \frac{1}{b} - \frac{1}{a} - \frac{1}{c} <
\frac{1}{b} - \frac{1}{a} \\
& \Leftrightarrow &
\frac{2}{b} - \frac{1}{a} - \frac{1}{c} <
\frac{1}{b} - \frac{b}{(a+c)b - ac} < \frac{1}{b} - \frac{1}{a} \\
& \Leftrightarrow &
b > \frac{ac}{a+c-\sqrt{ac}}.
\end{eqnarray*}
Thus, $\lambda_1 \in (0,c)$ and $\lambda_2 \in (b,a)$
if and only if $c < (a-2b)ab/(a-b)^2$ and $b > ac/(a+c-\sqrt{ac})$.
\end{proof}

Let us look for the winding numbers of the trajectories described in
the previous proposition.
The winding numbers $m_0$, $m_1$, and $m_2$ describe how the
periodic billiard trajectories fold in $\Rset^3$.
The following results can be found in~\cite[table 1]{CasasRamirez2011}.
First, $m_0$ is the period.
Second, $m_1$ is the number of $xy$-crossings and
$m_2$ is twice the number of turns around the $z$-axis for EH1-caustics;
$m_1$ is twice the number of turns around the $x$-axis and
$m_2$ is the number of $yz$-crossings for EH2-caustics;
$m_1$ is the number of tangential touches with each hyperboloid of one sheet
caustic and $m_2$ is twice the number of turns around the $z$-axis for H1H1-caustics;
$m_1$ is the number of $xz$-crossings and
$m_2$ is the number of $yz$-crossings for H1H2-caustics.
Besides, all periodic trajectories with H1H1-caustics or H1H2-caustics
have even period.
Several periodic billiard trajectories with elliptic period $m = 3$
were depicted in~\cite[Table XV and Table XVII]{CasasRamirez2012}.
We conclude by direct inspection of those pictures
that the nonsingular periodic billiard trajectories inside a
triaxial ellipsoid with elliptic period $m = 3$ have winding numbers
\[
m_2 = 2,\qquad
m_1 = 4,\qquad
m_0 = 6.
\]
This agrees with the formulas~(\ref{eq:WindingNumbers_C_n_n})
given in Theorem~\ref{thm:Main}.
We emphasize that those formulas were not rigorously proved,
because their ``proof'' was based on Conjecture~\ref{con:WindingNumbers}.

Next, we establish the algebraic formulas for the caustic parameters of
other nonsingular periodic billiard trajectories.
We begin with a technical lemma about four-degree polynomials.

\begin{lem}\label{lem:Technical}
Let $Q(x) = (x-\alpha_-)(x-\beta_-)(x-\beta_+)(x-\alpha_+)$ for some
$\alpha_- < \beta_- < \beta_+ < \alpha_+$.
Let $\nu_- \in (\alpha_-,\beta_-)$, $\nu \in(\beta_-,\beta_+)$,
and $\nu_+ \in (\beta_+,\alpha_+)$ be the three roots of $Q'(x)$.
Then
\[
Q(\nu_+) < Q(\nu_-) \Leftrightarrow \alpha_- + \alpha_+ > \beta_- + \beta_+.
\]
\end{lem}

\begin{proof}
If we set $\eta = (\beta_+ + \beta_-)/2$ and
$\xi = (\beta_+ - \beta_-)/2$, then
\[
Q(\eta + s) - Q(\eta - s) = 2(s^2-\xi^2)(\beta_- + \beta_+ - \alpha_- - \alpha_+)s,
\qquad \forall s \in \Rset.
\]
On the one hand,
if $\alpha_- + \alpha_+ > \beta_- + \beta_+$,
then $Q(\eta + s) < Q(\eta - s)$ for all $s > \xi$,
which implies that $Q(\nu_+) < Q(\nu_-)$.
On the other hand,
if $\alpha_- + \alpha_+ < \beta_- + \beta_+$,
then $Q(\eta + s) > Q(\eta - s)$ for all $s > \xi$,
which implies that $Q(\nu_+) > Q(\nu_-)$.
Finally,
if $\alpha_- + \alpha_+ = \beta_- + \beta_+$, then
$Q(\eta + s) = Q(\eta - s)$ for all $s \in \Rset$,
which implies that $Q(\nu_+) = Q(\nu_-)$.
\end{proof}

We can now answer some questions about the nonsingular
periodic billiard trajectories found in the second item of
Theorem~\ref{thm:Samples}, although the study is restricted
to the spatial case.

\begin{pro}\label{pro:C43_H1H1}
There exist periodic billiard trajectories inside the triaxial
ellipsoid~(\ref{eq:Ellipsoid_3D}) with elliptic period $m=4$,
signature $\tau = (0,0,1)$, and caustic type H1H1 if and only if
\[
c < ab/(a+b).
\]
Besides, the caustic parameters $\lambda_1$ and $\lambda_2$ of such
periodic billiard trajectories are the roots of the quadratic polynomial
$(s^2_1 - s_2) t^2/2 - s_1 t + 1$, where
\begin{equation}\label{eq:sl}
s_l = 1/a^l + 1/b^l + 1/c^l - 2/d^l,\qquad l=1,2,
\end{equation}
and $d$ is the only root of the cubic polynomial
$t^3 - 2(a+b+c) t^2 + 3(ab+ac+bc) t - 4abc$
in the interval $(a,+\infty)$.
\end{pro}

\begin{proof}
If $\tau = (0,0,1)$,
the decompositions presented in Definition~\ref{defi:Decompositions}
are $J_\tau = \{2,3\}$, $K_\tau = \{1,4,5\}$, $V_\tau = \{1\}$,
and $W_\tau = \emptyset$.
Thus, $\mathcal{C}(4,3;\tau)$ holds if and only if there exists
some $\delta_1 \in (0,\gamma_5)$ such that the following two
equivalent properties hold:
\begin{enumerate}[(i)]
\item\label{item:Decompositions-first}
$P(x) = (x-\delta_1)^2(x-\gamma_2)(x-\gamma_3) \Rightarrow
Q(x) := x(x-\gamma_1)(x-\gamma_4)(x-\gamma_5) = P(x) - P(0)$.
\item\label{item:Decompositions-second}
$\gamma^l_2 + \gamma^l_3 + 2\delta_1^l = \gamma_1^l + \gamma_4^l + \gamma_5^l$, for $l=1,2,3$.
\end{enumerate}

If the caustic type is H1H1, then $0 < c < \lambda_1 < \lambda_2 < b < a$, so
\[
\gamma_1 = 1/c,\qquad
\gamma_2 = 1/\lambda_1,\qquad
\gamma_3 = 1/\lambda_2,\qquad
\gamma_4 = 1/b,\qquad
\gamma_5 = 1/a.
\]
Let $\rme_l = \rme_l(\gamma_1,\gamma_4,\gamma_5)$ for $l = 1,2,3$.
We set $d = 1/\delta_1 > a$.
Then $Q(x) = x^4 - \rme_1 x^3 + \rme_2 x^2 - \rme_1 x$ and
$\delta_1$ is a root of $Q'(x) = 4x^3 - 3 \rme_1 x^2 + 2 \rme_2 x - \rme_3$.
Hence, $d$ is a root of the cubic polynomial
\[
q(t) = -abc t^3 Q'(1/t) = t^3 - 2(a+b+c) t^2 + 3(ab+ac+bc) t - 4abc.
\]
We note that $q(0) = -4abc < 0$, $q(b) = -b(b-a)(b-c) > 0$,
$q(a) = -a(a-b)(a-c) < 0$, and $\lim_{t \to +\infty} q(t) = +\infty$.
This shows that $q(t)$ has just one root in the interval $(a,+\infty)$.

From property~(\ref{item:Decompositions-second}) above,
we deduce that the sums $s_l := 1/\lambda_1^l + 1/\lambda_2^l$ verify
relations~(\ref{eq:sl}).
Besides, $1/\lambda_1$ and $1/\lambda_2$ are the roots of
$(x-1/\lambda_1)(x-1/\lambda_2) = x^2 - s_1 x + (s_1^2-s_2)/2$,
so that $\lambda_1$ and $\lambda_2$ are the roots of the quadratic polynomial
$(s^2_1 - s_2) t^2/2 - s_1 t + 1$.

We look for ellipsoidal parameters such that the previous periodic
trajectories exist.
From property~(\ref{item:Decompositions-first}) above,
we deduce that such ellipsoidal parameters exist
if and only the graph $\{ y = Q(x) \}$
intersects the horizontal line $\{ y = Q(\delta_1)\}$
at two different points $\gamma_2,\gamma_3 \in (\gamma_4,\gamma_5)$
or, equivalently, if and only if $Q(\delta_3) < Q(\delta_1)$,
where $\delta_1 < \delta_2 < \delta_3$ are the three ordered roots of
the derivative of the polynomial $Q(x) = x(x-\gamma_1)(x-\gamma_4)(x-\gamma_5)$.
But
\[
Q(\delta_3) < Q(\delta_1) \Leftrightarrow
\gamma_1 > \gamma_4 + \gamma_5 \Leftrightarrow
c < ab/(a+b),
\]
according to Lemma~\ref{lem:Technical}.
\end{proof}

As we have explained before,
the period and winding numbers of any nonsingular periodic billiard
trajectory can be determined by direct inspection of its corresponding figure.
A periodic billiard trajectory with elliptic period $m = 4$ and caustic type H1H1
whose caustic parameters verify the relations given in Proposition~\ref{pro:C43_H1H1}
is displayed in~\cite[Figure 13]{CasasRamirez2011}.
That trajectory has (Cartesian) period $m_0 = 4$ and winding numbers
\[
m_2 = 2,\qquad  m_1 = 3,\qquad m_0 = 4.
\]
Hence, $\gcd(m_0,m_1,m_2) = 1$,
so the elliptic winding numbers are
$\widetilde{m}_2 = 2$, $\widetilde{m}_1=3$, and $\widetilde{m}_0=4$.
This result reinforces Conjecture~\ref{con:tau}.
Besides,
these nonsingular four-periodic billiard trajectories with caustic type H1H1
are quite interesting,
because they display the minimal period among all nonsingular periodic
billiard trajectories; see~\cite[Theorem 1]{CasasRamirez2011}.

We end the study at this point.
We just mention that there exist similar results when the signature or the
caustic type do not coincide with the ones given in Proposition~\ref{pro:C43_H1H1}.
Analogously, the case $m = 5$ can be dealt with using the same techniques,
although the final formulas become more complicated.
For instance,
it can be easily checked that the caustic parameters $\lambda_1$ and $\lambda_2$
of the billiard trajectories with elliptic period $m = 5$
and signature $\tau = (1,1,0)$ verify the homogeneous symmetric
polynomial equations
\[
8 s_3 + s_1^3 = 6 s_1 s_2,\qquad
16 s_4 + s_1^4 = 4s_1^2s_2 + 4 s_2^2,
\]
where $s_l = 1/a^l + 1/b^l + 1/c^l + 1/\lambda_2^l + 1/\lambda_1^l$
for $l=1,2,3,4$.
Each of the other signatures $\tau \in \mathcal{T}(5,3)$
gives rise to similar homogeneous ---although not symmetric--- polynomial
equations of degrees three and four in the variables $1/a$, $1/b$, $1/c$,
$1/\lambda_1$ and $1/\lambda_2$.
We left the details to the reader.
Finally, we remember that the original matrix formulation of
the generalized Cayley condition $\mathcal{C}(5,3)$ gives rise to two
homogeneous symmetric polynomial equations of degrees 23 and 24
in those five variables, as explained in Section~\ref{sec:Practical}.
This confirms, once again, that the polynomial formulation
offers great computational advantages over the matrix formulation.

\end{document}